\documentclass[12pt,a4paper]{article}
\usepackage[latin1]{inputenc}
\usepackage[english]{babel}
\usepackage{hyperref}
\usepackage{amsfonts}
\usepackage{amssymb}
\usepackage{amsmath}
\usepackage{amsthm}
\usepackage[left=2.8cm,right=2.8cm,top=2cm,bottom=2cm]{geometry}
\usepackage{dsfont}
\usepackage{enumitem}
\usepackage{underscore}

\usepackage{stmaryrd}
\usepackage{mathtools}

\usepackage{multirow}
\usepackage{array}

\newcolumntype{P}[1]{>{\centering\arraybackslash}p{#1}}

\usepackage{tikz}
\usepackage{pgfplots}
\pgfplotsset{compat=newest}

\definecolor{dunkelblau}{rgb}{0,0,0.7}
\hypersetup{colorlinks=true,urlcolor=dunkelblau,linkcolor=dunkelblau,citecolor=dunkelblau}

\newtheoremstyle{standard}{0.7em}{0.7em}{\itshape}{}{\bfseries}{.}{.5em}{}
\theoremstyle{standard}
\newtheorem{lem}{Lemma}[section]
\newtheorem{prop}[lem]{Proposition}
\newtheorem{thm}[lem]{Theorem}
\newtheorem{cor}[lem]{Corollary}

\newtheoremstyle{definitio}{0.7em}{0.7em}{}{}{\bfseries}{.}{.5em}{}
\theoremstyle{definitio}
\newtheorem{defi}[lem]{Definition}  
 
\newtheorem{ex}[lem]{Example} 
\newtheorem{rem}[lem]{Remark}
 
\DeclareMathOperator{\Spec}{Spec}
\DeclareMathOperator{\rank}{rank}
\DeclareMathOperator{\Dic}{Dic}
\DeclareMathOperator{\Aut}{Aut}
\DeclareMathOperator{\mex}{mex}
\DeclareMathOperator{\chara}{char}

\usepackage{mathrsfs}
\renewcommand{\P}{\mathscr{P}}
\newcommand{\N}{\mathscr{N}}

\newcommand{\Z}{\mathds{Z}}
\newcommand{\IN}{\mathds{N}}
\newcommand{\IF}{\mathds{F}}

\newcommand{\IA}{\mathds{A}}

\newcommand{\I}{\stackrel{\mathrm{I}}{\leadsto}}
\newcommand{\II}{\stackrel{\mathrm{II}}{\leadsto}}

\newcommand{\clos}[1]{\langle\kern-.13em\langle#1\rangle\kern-.13em\rangle}

\newcommand{\snr}{250} 

\let\oldbibliography\thebibliography
\renewcommand{\thebibliography}[1]{\oldbibliography{#1}
\setlength{\itemsep}{0.5ex}}

\begin{document}

\title{Algebraic games -- Playing with groups and rings}
\author{Martin Brandenburg\footnote{brandenburg@uni-muenster.de}}
 
\maketitle

\begin{abstract}
\noindent Two players alternate moves in the following impartial combinatorial game: Given a finitely generated abelian group $A$, a move consists of picking some $0 \neq a \in A$. The game then continues with the quotient group $A/\langle a \rangle$. We prove that under the normal play rule, the second player has a winning strategy if and only if $A$ is a square, i.e.\ $A \cong B \times B$ for some abelian group $B$. Under the mis\`{e}re play rule, only minor modifications concerning elementary abelian groups are necessary to describe the winning situations. We also compute the nimbers, i.e.\ Sprague-Grundy values of $2$-generated abelian groups. An analogous game can be played with arbitrary algebraic structures. We study some examples of non-abelian groups and commutative rings such as $R[X]$, where $R$ is a principal ideal domain.
\end{abstract}

\section{Introduction}

Consider the following two-person impartial combinatorial game: Given an abelian group $A$, a move consists of picking some $0 \neq a \in A$ and replacing $A$ by the quotient group $A/\langle a \rangle$; here $\langle a \rangle$ denotes the subgroup generated by $a$. Hence, the next move consists of picking some $0 \neq \overline{b} \in A/\langle a \rangle$ and replacing $A/\langle a \rangle$ by $A/\langle a \rangle / \langle \overline{b} \rangle \cong A/\langle a,b \rangle$, etc. Under the normal (resp.\ mis\`{e}re) play rule, the player with the last possible move wins (resp.\ loses): When $A=0$, the next player cannot move and therefore wins under the mis\`{e}re play rule and loses under the normal play rule. The ending condition is satisfied precisely when $A$ is finitely generated. For which $A$ does the first player have a winning strategy, i.e.\ when is $A$ an $\N$-position? And for which $A$ does the second player have a winning strategy, i.e.\ when is $A$ a $\P$-position? This question can be asked both for the normal as well as for the mis\`{e}re play rule.
 
The moves in the game starting with a finitely generated abelian group $A$ may also be described by a sequence of elements $a_1,a_2,a_3,\dotsc$ of $A$ such that $a_i$ is not contained in the subgroup $\langle a_1,\dotsc,a_{i-1} \rangle$ generated by the previous elements. In fact, the moves are then given by
\[A \I A/\langle a_1 \rangle \II A/\langle a_1,a_2 \rangle \I A/\langle a_1,a_2,a_3 \rangle \II \dotsc.\]
The game ends as soon as $a_1,\dotsc,a_i$ generate $A$. Hence, our game features some similarities with the game considered in \cite{AH,BES}, where the weaker condition $a_i \notin \{a_1,\dotsc,a_{i-1}\}$ but the same ending condition were imposed. In this setup, the games of two groups $A$ and $B$ are already equivalent if there is a bijection between the underlying sets of $A$ and $B$ which induces a bijection between the maximal subgroups. This is not the case for our game, which seems to incorporate better the specific algebraic structure of $A$ and does not put any emphasis on the underlying set of $A$.
   
Our main theorem, proven in Section \ref{sec:ab}, states the following:
 
\begin{thm}
Let $A$ be a finitely generated abelian group.
\begin{itemize}[leftmargin=4ex]
\item Under the normal play rule, $A$ is a $\P$-position if and only if $A$ is a square, i.e.\ $A \cong B^2$ for some abelian group $B$.
\item Under the mis\`{e}re play rule, $A$ is a $\P$-position if and only if $A$ is
\begin{itemize}[leftmargin=4ex]
\item either a square, but not isomorphic to $(\Z/p)^s$ for some prime $p$ and some even number $s$,
\item or isomorphic to $(\Z/p)^s$ for some prime $p$ and some odd number $s$.
\end{itemize}
\end{itemize}
\end{thm}

Here, $\Z/p$ abbreviates $\Z/\langle p \rangle$. Recall that groups of the form $(\Z/p)^s$ are also called elementary abelian $p$-groups. They cause the only difference between the normal and the mis\`{e}re play rule. Notice that the game of a product of abelian groups $A \times A'$ \emph{does not} equal the sum of the games of $A$ and $A'$, and that usually $(A \times A')/\langle (a,a') \rangle$ is \emph{not isomorphic} to $A/\langle a \rangle \times A'/\langle a'\rangle$. This is why the theorem cannot be proven as easily as one might guess at first glance.

A similar theorem holds for finitely generated modules over a principal ideal domain. Here, we quotient out cyclic submodules.
 
Our proof is constructive and will include a winning strategy (see Theorem \ref{fin}). For example, the abelian group $\Z/4 \oplus \Z/8$ is a normal $\N$-position: Player I quotients out the element $0 \oplus 4$ to obtain $\Z/4 \oplus \Z/4$. No matter what Player II does, he will produce an abelian group isomorphic to $\Z/4$ or $\Z/2 \oplus \Z/4$. In the first case Player I quotients out the generator of $\Z/4$ and wins. In the second case Player I quotients out $0 \oplus 2$, so that Player II gets $\Z/2 \oplus \Z/2$. He can only react with an abelian group isomorphic to $\Z/2$. Player I quotients out the generator and therefore wins. The group $\Z/4 \oplus \Z/8$ is also a mis\`{e}re $\N$-position: From $\Z/4$ Player I quotients out $2$, and from $\Z/2 \oplus \Z/4$ he quotients out $0 \oplus 1$. In each case Player II has to play with $\Z/2$ and does the last move, so that he loses under the mis\`{e}re play rule.

We will also compute the nimbers, i.e.\ Sprague-Grundy values, of some finitely generated abelian groups; recall that the nimber of an impartial combinatorial game $G$ is the unique ordinal number $\alpha$ for which $G$ is equivalent to the Nim game $\ast \alpha$ with one pile of size $\alpha$ \cite[Chapter 11]{C}. Specifically, the nimber of a finitely generated abelian group $A$ is recursively defined as the least ordinal number which does not equal the nimber of any quotient $A/\langle a \rangle$, where $0 \neq a \in A$.
 
\begin{thm}
If $n \geq 1$, then the nimber of $\Z/n$ equals the number $\Omega(n)$ of prime factors of $n$ counted with multiplicity. The nimber of $\Z$ equals the first infinite ordinal number $\omega$. The nimber of $\Z/n \oplus \Z$ equals $\omega + \Omega(n)$.
\end{thm}
 
\begin{thm} \label{jyrki}
Let $p$ be a prime number and $0 \leq n \leq m$ be natural numbers. Let $k \coloneqq m-n$ and $\Delta_k  \coloneqq  \frac{1}{2} k(k+1)$ be the triangular number. The nimber of $\Z/p^n \oplus \Z/p^m$ equals
\[\left\{\begin{array}{cc} n+m &\text{if } n \leq \Delta_k, \\ \Delta_k + (n-\Delta_k-1 \bmod k+1) &\text{if } n > \Delta_k. \end{array}\right.\]
\end{thm}

See Figure \ref{tabelle} in Section \ref{subsec:nimb} for how these numbers look like. The nimber of an arbitrary $2$-generated finite abelian group $\Z/n_1 \oplus \Z/n_2$ with $n_1 \mid n_2$ turns out to be the nimber of the $p$-group $\Z/p^{\Omega(n_1)} \oplus \Z/p^{\Omega(n_2)}$ for any chosen prime number $p$.
  
It is possible to generalize the game to arbitrary algebraic structures of some given signature, as we shall explain in Section \ref{sec:game}. For example, if we start with a group $G$, then a move consists of replacing $G$ by the quotient group $G / \clos{a}$, where $1 \neq a \in G$ and $\clos{a}$ denotes the normal subgroup generated by $a$. We briefly analyze this game in Section \ref{sec:grp}. We also look at the related game of subgroups of a group $G$, which is more balanced as to the proportion of $\N$- and $\P$-positions. Here we start with the trivial subgroup of $G$ and a move replaces a subgroup $U$ of $G$ by the subgroup $\langle U,g\rangle$ for some $g \in G \setminus U$. We only make some initial considerations such as the following two results.
 
\begin{prop}
Let $n \geq 1$. In the game of subgroups, the dihedral group $D_n$ is a normal $\P$-position if and only if $n$ is a prime number. The symmetric group $S_n$ is a normal $\P$-position if and only if $n \neq 2$.
\end{prop}
 
Commutative rings provide another very interesting class of algebraic structures to play with. Starting with a commutative ring $R$, a move consists of picking some $0 \neq a \in R$ and replacing $R$ by the quotient ring $R/\langle a \rangle$, where $\langle a \rangle$ denotes the ideal generated by $a$. The ending condition is satisfied precisely for Noetherian commutative rings. Since every non-trivial commutative ring $R$ has a move to the trivial ring by taking $a \coloneqq 1$, it is reasonable to play this game under the mis\`{e}re play rule. This game has been popularized by Will Sawin (\texttt{http://mathoverflow.net/questions/93276}), although it may have been mathematical folklore much earlier. Using the duality between commutative rings and affine schemes \cite{GW}, it can be seen as a geometric game. We will analyze it in Section \ref{sec:ring}. The main results are the following:

\begin{prop}
Let $R$ be a Noetherian commutative ring which is a mis\`{e}re $\P$-position. Then $R$ cannot be written as a product of two non-trivial rings. In other words, $R$ does not contain any non-trivial idempotent elements.
\end{prop}
  
\begin{thm}
Let $R$ be a principal ideal domain, which is not a field.
\begin{itemize}[leftmargin=4ex]
\item If $p \in R$ is a prime element, then $R/\langle p \rangle$ is a mis\`{e}re $\P$-position. Hence, $R$ is a mis\`{e}re $\N$-position.
\item The ring $R[X]/\langle X^2 \rangle$ is a mis\`{e}re $\P$-position. Hence, $R[X]$ is a mis\`{e}re $\N$-position.
\end{itemize}
\end{thm}

This implies for example that the polynomial ring $K[X,Y]$ is a mis\`{e}re $\N$-position, where $K$ is a field. If $K$ is algebraically closed, we will provide alternative proofs for this fact by showing that $K[X,Y]/\langle Y^2-X^3-1\rangle$, the coordinate ring of an elliptic curve, and $K[X,Y]/\langle Y^2-X^3 \rangle$, the coordinate ring of a cuspidal cubic curve, are both mis\`{e}re $\P$-positions. We will also compute the nimbers of some commutative rings in Section \ref{subsec:nimbring}.
 
The games introduced in this paper might be called \textit{algebraic games} in contrast to the well-studied topological games \cite{T87}. By the very nature of these games, we frequently use backward induction. For example, in the game of commutative rings we have to go all the way down to smaller and smaller zero-dimensional rings in order to solve the game for more interesting rings such as $K[X,Y]$. Algebraic games can be fun, but they also require a deeper understanding of how algebraic structures are built up from smaller ones. Moreover, the nimber of an algebraic structure is an interesting new ordinal invariant whose computation may be connected with their classification.
 
Several interesting questions about algebraic games are yet to be answered, for example how to compute the nimbers of arbitrary finitely generated abelian groups, if there is any geometric description of those affine varieties whose coordinate rings are mis\`{e}re $\P$-positions in the game of commutative rings, and how to determine the nimbers of polynomial rings.

\paragraph*{Acknowledgements}
 
For various discussions and suggestions on the game of rings I would like to thank Will Sawin and Kevin Buzzard. Special thanks goes to Diego Montero who corrected some errors in a preliminary version and simplified the proof of Proposition \ref{compute}. I would like to thank Jyrki Lahtonen for suggesting the formula in Theorem \ref{jyrki}. Finally I would like to thank most sincerely Bernhard von Stengel and the anonymous referees for their numerous useful and valuable suggestions for improvement.

 
\section{The game in general} \label{sec:game}

\subsection{Basics of combinatorial game theory}

In this subsection we  briefly recall some basic notions of combinatorial game theory. For details we refer to textbooks such as \cite{ANW,BCG,C,S13}.

We only consider two-person impartial combinatorial games. This means that Player I (who starts) and Player II alternate in making moves, and each player has the same set of options (possible moves) for a given position in the game. No chance moves are involved, the game is purely combinatorial. Every game has a set of terminal positions. We require the ending condition, which asserts that the game has to end after some finite number of moves. However, we allow infinitely many positions. Formally, a game may be defined just as a well-founded set, the options being the elements of this set, which are games themselves.
 
The first player who cannot move loses under the \textit{normal play rule}. He wins under the \textit{mis\`{e}re play rule}. Thus, under the normal play rule one wants to be the last player to move, whereas under the mis\`{e}re play rule one actually wants to prevent this. Often mis\`{e}re games are more complicated than normal ones.
 
We call a position in the game an \emph{$\N$-position} if the next player to move has a winning strategy. If the previous player has a winning strategy, we call it a \emph{$\P$-position}. This definition applies to both play rules. We also use $\N$ and $\P$ as adjectives. One of the first basic observations in combinatorial game theory is the following: Under both play rules, every position is either an $\N$-position or a $\P$-position. In fact, we can declare a position to be $\N$ or $\P$ recursively as follows:

\begin{enumerate}[leftmargin=4ex]
\item Every terminal position is a normal $\P$-position (resp.\ mis\`{e}re $\N$-position).
\item A non-terminal position is normal (resp.\ mis\`{e}re) $\N$, when \textit{some} option from it is a normal (resp.\ mis\`{e}re) $\P$-position.
\item A position is normal (resp.\ mis\`{e}re) $\P$, when \textit{every} option is a normal (resp.\ mis\`{e}re) $\N$-position.
\end{enumerate}

Intuitively, 1.\ declares the play rule, 2.\ asserts the existence of a winning move for $\N$-positions, and 3.\ denies it for $\P$-positions. The ending condition easily implies:
 
\begin{prop} \label{charNP}
Under either play rule, the sets of $\P$- and $\N$-positions are characterized by the three properties above.
\end{prop}

\begin{ex}
Consider the game Nim with just two piles: We have two piles of counters. A move reduces the number of counters in exactly one of the piles. Under the normal play rule, $(x,y)$ is a $\P$-position if and only if $x=y$, i.e.\ $(x,y)$ is a ``square''. In fact, 1.\ the terminal position $(0,0)$ is a square, 2.\ every non-square can be moved to some square, and 3.\ squares cannot move to squares. Under the mis\`{e}re play rule, the $\P$-positions are \textit{almost} the same: $(0,0)$ and $(1,1)$ are mis\`{e}re $\N$, and $(1,0)$ and $(0,1)$ are mis\`{e}re $\P$, but the rest is as before. We have mentioned this example since the game of abelian groups will be similar, although much more complicated.
\end{ex}

\begin{rem} \label{SG}
If $\alpha$ is any ordinal number, then $\ast \alpha$ denotes the Nim game with one pile of size $\alpha$. By definition its options are the Nim games $\ast \beta$ with $\beta<\alpha$. The Sprague-Grundy Theorem states that every combinatorial game $G$ under the normal play rule is equivalent to a Nim game $\ast \alpha(G)$ for some unique ordinal number $\alpha(G)$, called the \emph{nimber} of $G$. This is an ordinal number which may be defined recursively by
\[\alpha(G) = \mex \{\alpha(H) : H \text{ is an option of } G\}.\]
Here, $\mex(S)$ denotes the smallest ordinal number not contained in $S$. For example, one has $\mex(\{1,3\})=0$, $\mex(\{0,2\})=1$ and $\mex(\{0,1,2,\dotsc\})=\omega$. Observe that $\alpha(G)=0$ holds if and only if $G$ is a normal $\P$-position. Otherwise, we have $\alpha(G)>0$. The nimber of $G$ carries much more information than just the knowledge about which player wins. It is important to know this nimber when $G$ is played in a sum of games. 
\end{rem}
 
\subsection{The game of algebraic structures}

Now let us introduce the game of algebraic structures. Before we define it in full generality, we will define it in the special cases of abelian groups, groups and rings.

\begin{defi}
Let $A$ be an abelian group. The positions in the \emph{game of $A$} are abelian groups again. The initial position is $A$ itself, the terminal positions are the trivial groups. A move from an abelian group $B$ consists of picking some $0 \neq b \in B$ and replacing $B$ by the quotient abelian group $B / \langle b \rangle$, where $\langle b \rangle$ denotes the cyclic subgroup generated by $b$. Thus, the options of $B$ are the quotient groups $B/C$, where $C$ is a non-trivial cyclic subgroup of $B$.
\end{defi}
 
\begin{defi}
Let $G$ be a group. The positions in the \emph{game of $G$} are groups again. The initial position is $G$ itself, the terminal positions are the trivial groups. A move from a group $H$ consists of picking some $0 \neq h \in H$ and replacing $H$ by the quotient group $H / \clos{h}$, where $\clos{h}$ denotes the normal subgroup generated by $h$. In other words, $\clos{h}$ is the subgroup generated by the conjugates $\{x h x^{-1} : x \in H\}$.
\end{defi}

\begin{defi}
Let $R$ be a ring; by definition rings are unital. The positions in the \emph{game of $R$} are rings again. The initial position is $R$ itself, the terminal position are the trivial rings. A move from a ring $S$ consists of picking some $0 \neq s \in S$ and replacing $S$ by the quotient ring $S / \langle s \rangle$, where $ \langle s \rangle$ denotes the ideal generated by $s$.
\end{defi}
 
Let us explain these games in more detail. Starting with an abelian group $A$, Player I picks some $0 \neq a \in A$ and gives $A/\langle a \rangle$ to Player II. The latter has to choose some element $\overline{b} \in A / \langle a \rangle$ and gives $A/\langle a \rangle / \langle \overline{b} \rangle$ to Player I. If $b \in A$ denotes a preimage of $\overline{b} \in A/\langle a \rangle$, then the group $A/\langle a \rangle / \langle \overline{b} \rangle$ is isomorphic to $A/\langle a,b \rangle$, so that we might as well continue with this group. In fact, it is a general observation that two isomorphic abelian groups have equivalent games. The condition $\overline{b} \neq 0$ means $b \notin \langle a \rangle$. The next move is given by some element $c \notin \langle a,b \rangle$ which produces the abelian group $A/\langle a,b,c \rangle$. In general, after the $i$th move we have an abelian group of the form $A/\langle a_1,\dotsc,a_i \rangle$, and for each $j \leq i$ the element $a_j$ is not contained in $\langle a_1,\dotsc,a_{j-1} \rangle$. Under the normal play rule, the next player loses when there is no element $\neq 0$ anymore, i.e.\ $A=\langle a_1,\dotsc,a_i \rangle$. Under the mis\`{e}re play rule, he would win.
 
Instead of this iterative description, observe that the game of $A$ is just defined recursively by the property that its options are the games of the quotients $A/\langle a \rangle$, where $0 \neq a \in A$. This recursive description will turn out to be quite useful.
 
As a trivial example, we note that $\Z$ is an $\N$-position in the game of abelian groups, in fact under both play rules. Under the normal play rule, Player I chooses the element $a \coloneqq 1$ and returns the trivial group $\Z/1$ to Player II, who loses immediately. Under the mis\`{e}re play rule, Player I may choose any prime number, for instance $a \coloneqq 7$, because the quotient group $\Z/7$ can only be moved to the trivial group by Player II.
 
Let us verify that the ending condition is satisfied precisely for the Noetherian abelian groups; recall that a group is called Noetherian if there is no infinite strictly increasing chain of subgroups of $A$ \cite[Chapter 6]{AM69}. Using the iterative description of the game of $A$, it is clear that an infinite sequence of moves produces such an infinite strictly increasing chain of subgroups of $A$. Conversely, if $A_0 \subsetneq A_1 \subsetneq A_2 \subsetneq \cdots$ is such a chain of subgroups of $A$, then we may choose elements $a_i \in A_i \setminus A_{i-1}$ for $i \geq 1$, and these constitute an infinite sequence of moves in the game of $A$. But an abelian group $A$ is Noetherian if and only if $A$ is finitely generated; this follows since $\Z$ is a Noetherian ring \cite[Proposition 6.5]{AM69}. Thus, we have proven:
 
\begin{prop}
The game of an abelian group $A$ satisfies the ending condition if and only if $A$ is finitely generated.
\end{prop}
  
Similar remarks apply to the games of groups and rings: The moves in the game of a group $G$ may be described by elements $a_1,a_2,\dotsc$ in $G$ such that $a_j$ is not contained in the normal subgroup $\clos{a_1,\dotsc,a_{j-1}}$ generated by the previous elements. The ending condition is satisfied precisely when there is no infinite strictly increasing chain of normal subgroups of $G$; this property may hold even if the group is not finitely generated. Similarly, the moves in the game of a ring $R$ may be described by elements $a_1,a_2,\dotsc$ in $R$ such that $a_j$ is not contained in the ideal $\langle a_1,\dotsc,a_{j-1}\rangle$ generated by the previous elements. The ending condition is satisfied precisely when $R$ is Noetherian, i.e.\ when there is no infinite strictly increasing chain of ideals of $R$. Notice that every non-trivial ring $R$ has a move to the zero ring $R/\langle 1 \rangle = 0$. In other words, all non-trivial rings are normal $\N$-positions. It is more challenging to determine the mis\`{e}re $\N$-positions.
 
In order to study and prove some of the properties of the games of groups, abelian groups and rings at once, and even further examples of algebraic structures not covered in this paper, it makes sense to unify these games with the help of universal algebra \cite{BS} as follows:
 
\begin{defi}
Given some algebraic structure $A$ of some signature \cite[II, \S 1]{BS}, i.e.\ a set equipped with various functions with various arities as prescribed by the signature, a move in the \emph{game of $A$} consists of choosing two elements $a \neq b$ in $A$ and replacing $A$ by the quotient structure $A / (a \sim b)$ of the same signature. This is defined to be $A/{\sim}$, where $\sim$ is the congruence relation on $A$ generated by $(a,b)$ \cite[II, \S 5]{BS}. The game ends as soon as $A$ has at most one element left.
\end{defi}
  
This specializes to the games mentioned before, since congruence relations on groups (resp.\ rings) correspond to normal subgroups (resp.\ ideals), and because $a=b$ holds in a group (resp.\ ring) if and only if $a b^{-1} = 1$ (resp.\ $a-b=0$) is satisfied.

If $R$ is a ring, then $R$-modules \cite[Chapter II]{AM69} provide another type of algebraic structures, which coincide with $R$-vector spaces when $R$ is a field. The game of $R$-modules is very similar to the game of abelian groups, except that we quotient out cyclic submodules. The ending condition is satisfied precisely for Noetherian $R$-modules, which coincide with finitely generated $R$-modules when $R$ is a Noetherian ring \cite[Proposition 6.5]{AM69}.
  
In the general case of an algebraic structure $A$, a sequence of moves consists of elements $(a_1,b_1)$, $(a_2,b_2)$, $\dotsc$, $(a_n,b_n)$ in $A \times A$ such that
\begin{enumerate}[noitemsep,leftmargin=4ex]
\item $a_i \neq b_i$
\item $a_i \sim b_i$ cannot be derived from $a_1 \sim b_1,\dotsc,a_{i-1} \sim b_{i-1}$
\end{enumerate}
More formally, 2.\ means that $\overline{a_i} \neq \overline{b_i}$ holds in $A / (a_1 \sim b_1,\dotsc,a_{i-1} \sim b_{i-1})$. Even more formally, let $R_i$ be the congruence relation generated by $(a_1,b_1),\dotsc,(a_i,b_i)$. Then we have proper inclusions
\[\Delta(A) \subsetneq R_1 \subsetneq R_2 \subsetneq \cdots \subsetneq R_n \subseteq A \times A,\]
starting with the diagonal $\Delta(A) \coloneqq \{(a,a) : a \in A\}$ of $A$. As before, one verifies:
 
\begin{prop}
The game of an algebraic structure $A$ satisfies the ending condition if and only if $A$ does not contain an infinite strictly increasing chain of congruence relations.
\end{prop}
  
However, notice that the outcome of the game does not only depend on the partial order of congruence relations, because we cannot characterize principal ones in the language of partial orders. At least the following is true and easy to prove: 
 
\begin{prop}
If $A,B$ are isomorphic algebraic structures of the same signature, the corresponding games have the same outcome. In other words, $A$ is a $\P$-position (resp.\ an $\N$-position) if and only if $B$ is.
\end{prop}

We will use this result all the time. The following example illustrates that the game is easy to understand when some dimension or size classifies the whole structure:
 
\begin{ex} \label{vs}
Let us play with a vector space $V$ over a fixed field. The ending condition holds if and only if $V$ is finite-dimensional. The game only depends on the dimension of $V$. Every move reduces it by one. The terminal vector spaces are those of dimension zero. By induction it follows that $V$ is a normal $\P$-position if and only if its dimension is even. Otherwise it is a normal $\N$-position. The mis\`{e}re positions are vice versa.

Recall that an abelian group $A$ is called \emph{elementary abelian} if there is some prime $p$ with $pA=0$. This is equivalent to the condition that $A$ is a vector space over $\IF_p$. It follows that the finite abelian group $(\Z/p)^n$ is a normal $\P$-position if and only if $n$, its dimension over $\IF_p$, is even. This is the first piece of evidence for the main theorem about the game outcome of an arbitrary finitely generated abelian group in Section \ref{sec:ab}.
\end{ex} 
 
\subsection{Selective compound games}

In combinatorial game theory it is very useful to decompose games into sums of smaller games. The sum $G+H$ of two games $G,H$ is defined recursively by the property that the options of $G + H$ are $G' + H$ and $G + H'$, where $G'$ is an option of $G$ and $H'$ is an option of $H$. It is well-known that $G+H$ is a $\P$-position if and only if $G$ and $H$ are equivalent, i.e.\ $G,H$ have the same nimber.

Therefore, it is tempting to study the game of an algebraic structure $A$ by writing $A$ as a product or sum of smaller structures. However, we have already mentioned in the introduction that the game of a direct sum $A \oplus B$ of two abelian groups $A,B$ is not the sum of the games of $A$ and $B$. This is because we can choose an element $a \in A$ \emph{and} an element $b \in B$ simultaneously, one or both of them being non-zero, and then make the move $(A \oplus  B)/\langle (a,b)\rangle$. Another issue is that this group is usually not isomorphic to $A/\langle a \rangle \oplus B/\langle b\rangle$. For example, $(\Z \oplus \Z)/\langle (2,2)\rangle$ is isomorphic to the infinite abelian group $\Z/2 \oplus \Z$, which is far from being isomorphic to $\Z/2 \oplus \Z/2$. However, in some situations, $(A \oplus  B)/\langle (a,b)\rangle$ \emph{is} isomorphic to $A/\langle a \rangle \oplus B/\langle b\rangle$, as we shall see below. In that case, an option in the game of $A \oplus B$ is one of the games $A' \oplus B$, $A \oplus B'$ or $A' \oplus B'$, where $A'$ (resp.\ $B'$) is an option of $A$ (resp.\ $B$). This leads us to the following notion of a selective compound of two or more games, which is due to Smith \cite[Sections 7 and 8]{S}.
 
\begin{defi}
If $G_1,\dotsc,G_n$ are finitely many games, we can play a new game $G_1 \vee \cdots \vee G_n$, called the \emph{selective compound} of $G_1,\dotsc,G_n$. A position in $G_1 \vee \cdots \vee G_n$ is a tuple of positions in the games $G_1,\dotsc,G_n$. A move consists of picking a non-empty subset of $G_1,\dotsc,G_n$ and making a move in each of the chosen games. If $G_i$ is already over, i.e.\ happens to be a terminal position, then of course we continue with $G_1 \vee \cdots \vee \widehat{G_i} \vee \cdots \vee G_n$, with $G_i$ being removed.

We can also describe $G_1 \vee \cdots \vee G_n$ recursively: The options of $G_1 \vee \cdots \vee G_n$ are $G'_1 \vee \cdots \vee G'_n$, where $G'_i$ is either equal to $G_i$ or an option of $G_i$, the latter happening for at least one $i$. Thus, the difference to the sum $G_1 + \cdots + G_n$ is that we are allowed to move in more than just one of the games.
\end{defi}
 
\begin{prop} \label{sel-normal}
The selective compound $G_1 \vee \cdots \vee G_n$ is a normal $\P$-position if and only if every $G_i$ is a normal $\P$-position.
\end{prop}

\begin{proof}
It is clear that $G_1 \vee \cdots \vee G_n$ is terminal if and only if every $G_i$ is terminal. Now, if every $G_i$ is a normal $\P$-position, then the options of $G_1 \vee \cdots \vee G_n$ are $G'_1 \vee \cdots \vee G'_n$, where either $G'_i=G_i$ is a normal $\P$-position or $G'_i$ is an option of $G_i$, which is therefore a normal $\N$-position. The latter happens for at least one $i$, so that some $G'_i$ is a normal $\N$-position. If, on the other hand, $G_1 \vee \cdots \vee G_n$ has a non-empty set of indices $i$ for which $G_i$ are normal $\N$-positions, for these indices we may choose options $G'_i$ of $G_i$ which are normal $\P$-positions. For the other indices, we let $G'_i \coloneqq G_i$. Thus, each $G'_i$ is a normal $\P$-position, and $G'_1 \vee \cdots \vee G'_n$ is an option of $G_1 \vee \cdots \vee G_n$.
 \end{proof}

\begin{prop} \label{sel-misere}
The selective compound $G_1 \vee \cdots \vee G_n$ is a mis\`{e}re $\P$-position if and only if either
\begin{itemize}[leftmargin=4ex]
\item all games except one, say $G_i$, are over (i.e.\ terminal), and $G_i$ is a mis\`{e}re $\P$-position,
\item at least two of the games are not over yet, and each $G_i$ is a normal $\P$-position.
\end{itemize}
\end{prop}
 
\begin{proof}
Let us call $G_1 \vee \cdots \vee G_n$ a $\P'$-position if it satisfies the condition in the proposition, i.e.\ every $G_i$ is normal $\P$ when at least two are not finished yet, or only one $G_i$ is still playing and mis\`{e}re $\P$. We have to prove that $\P'$ satisfies the defining properties of $\P$ in Proposition \ref{charNP}.

First of all, the terminal positions are not $\P'$. Next, every non-terminal position which is not $\P'$ has some (winning) move to a position which is $\P'$: If all games except for $G_i$ are finished, then we continue to play only with $G_i$, which is mis\`{e}re $\N$ and therefore has some move to a mis\`{e}re $\P$-position, which is therefore $\P'$ (or terminal). If at least two games are not finished yet, then some of the games is normal $\N$. Now move in every one of these normal $\N$ games to some normal $\P$-position and leave the normal $\P$ games untouched. Then we obtain the game $G'_1 \vee \cdots \vee G'_n$ where each $G'_i$ is normal $\P$. If at least two $G'_i$ are not finished yet, this is $\P'$ and we are done. Otherwise, every $G_i$ which is normal $\N$ can be ended in one move and the other ones are already finished. Now we choose the following winning move: Pick some $G_j$ which is normal $\N$. If it is mis\`{e}re $\P$, end all the other $G_i$ and keep $G_j$. If it is mis\`{e}re $\N$, choose some option $G'_j$ of $G_j$ which is mis\`{e}re $\P$, and combine this move with ending all other $G_i$. In each case, we arrive at a single active game which is mis\`{e}re $\P$ and therefore $\P'$.
 
Finally, we have to prove that a $\P'$-position cannot move to a $\P'$-position. This is clear when only one game is active. When at least two games are active, then every active $G_i$ is normal $\P$ and therefore cannot be ended in one move, but rather can only be moved to some normal $\N$-position $G'_i$. Thus, every option of the selective compound still has at least two active games, one of them being normal $\N$. Therefore, this option is not $\P'$.
\end{proof}

\begin{ex}
The selective compound $\ast n \vee \ast m$ of two Nim piles of sizes $n$ and $m$ is a normal $\P$-position if and only if $n=m=0$. It is a mis\`{e}re $\P$-position if and only if $(n,m) \neq (0,0)$.
\end{ex}

\begin{ex}
Consider the selective compound of the game of a finite-dimensional $K$-vector space (Example \ref{vs}) with the game of a finite-dimensional $L$-vector space, where $K,L$ are two fields. Thus, the options of $K^n \vee L^m$ are $K^{n-1} \vee L^m$ (if $n \geq 1$), $K^n \vee L^{m-1}$ (if $m \geq 1$) and $K^{n-1} \vee L^{m-1}$ (if $n,m \geq 1$). This is equivalent to the number-theoretic game on the lattice $\IN \times \IN$, where the options of $(n,m)$ are $(n-1,m)$, $(n,m-1)$ and $(n-1,m-1)$. According to Proposition \ref{sel-normal}, $(n,m)$ is a normal $\P$-position if and only if both $n$ and $m$ are even. According to Proposition \ref{sel-misere}, $(n,m)$ is a mis\`{e}re $\P$-position if and only if one of the following three cases occurs:
\begin{itemize}[noitemsep,leftmargin=4ex]
\item $n=0$ and $m$ is odd
\item $m=0$ and $n$ is odd
\item $n \geq 2$ and $m \geq 2$ are even
\end{itemize}
{\small \begin{center}
\begin{tabular}{r|*{6}{c}}
$n \backslash m$\!\! & 0 & 1 & 2 & 3 & 4 & 5 \\ \hline
0 & $\N$ & $\P$ & $\N$ & $\P$ & $\N$ & $\P$ \\
1 & $\P$ & $\N$ & $\N$ & $\N$ & $\N$ & $\N$ \\
2 & $\N$ & $\N$ & $\P$ & $\N$ & $\P$ & $\N$ \\
3 & $\P$ & $\N$ & $\N$ & $\N$ & $\N$ & $\N$ \\
4 & $\N$ & $\N$ & $\P$ & $\N$ & $\P$ & $\N$ \\
5 & $\P$ & $\N$ & $\N$ & $\N$ & $\N$ & $\N$
\end{tabular}
\end{center}}
This selective compound is equivalent to the game of finitely generated $K \times L$-modules. Putting $K=\IF_p$ and $L=\IF_q$ for two distinct prime numbers $p,q$, these modules have underlying abelian groups of the form $(\Z/p)^n \oplus (\Z/q)^m$, where $n,m \geq 0$. Thus, we have determined which of these abelian groups are normal (resp.\ mis\`{e}re) $\P$-positions. This is the second piece of evidence for the main theorem about the game outcome of an arbitrary finitely generated abelian group in Section \ref{sec:ab}.
\end{ex}
  
Now let us put the initial plan into action.
  
\begin{prop} \label{prod}
Consider some fixed signature of algebraic structures. Assume that for all algebraic structures $A_1,\dotsc,A_n$ and all $a,b \in A  \coloneqq  A_1 \times \cdots \times A_n$ the canonical homomorphism
\[A/(a \sim b)\to A_1/(a_1 \sim b_1) \times \cdots \times A_n / (a_n \sim b_n)\]
is an isomorphism. Then for all algebraic structures $A_1,\dotsc,A_n$ the game of the product $A_1 \times \cdots \times A_n$ is equivalent to the selective compound of the games of  $A_1,\dotsc,A_n$.
 
More generally, let us assume that $\mathcal{T}_1,\dotsc,\mathcal{T}_n$ are classes of algebraic structures which are stable under quotients and contain the terminal structures with one element. Assume that for all $A_i \in \mathcal{T}_i$ and $a,b \in A  \coloneqq  A_1 \times \cdots \times A_n$ the canonical homomorphism
\[A/(a \sim b)\to A_1/(a_1 \sim b_1) \times \cdots \times A_n / (a_n \sim b_n)\]
is an isomorphism. Then for all algebraic structures $A_i \in \mathcal{T}_i$ the game of $A$ is the selective compound game of the games of  $A_1,\dotsc,A_n$.
\end{prop}

\begin{proof}
This follows from the definitions. The requirement $a \neq b$ in the definition of a move means that $a_i \neq b_i$ for at least one $i$, i.e.\ that we move in at least one factor. 
 \end{proof}

\begin{cor} \label{proda}
In the situation of Proposition \ref{prod}, the product $A=A_1 \times \cdots \times A_n$ is
\begin{itemize}[leftmargin=4ex]
\item normal $\P$ if and only if every $A_i$ is normal $\P$
\item mis\`{e}re $\P$ if and only if all factors except one, say $A_i$, are terminal, and $A_i$ is mis\`{e}re $\P$, or at least two factors are non-terminal, and every $A_i$ is normal $\P$.
\end{itemize}
\end{cor}

\begin{proof}
This follows from Propositions \ref{sel-normal}, \ref{sel-misere} and \ref{prod}.
\end{proof}

We will apply this result to abelian groups in Section \ref{sec:ab}. For the moment, we record an application to commutative rings.
 
\begin{ex} \label{ringprod} 
Let $R_1,\dotsc,R_n$ be commutative rings and let $R$ denote their product. Then for every $a \in R$ the induced homomorphism
\[R/\langle a \rangle \to R_1/\langle a_1 \rangle \times \cdots \times R_n/\langle a_n \rangle\]
is an isomorphism. This is because $\langle a \rangle$ also contains $e_i a = a_i e_i$, where $e_i$ denotes the idempotent element $(0,\dotsc,1,\dotsc,0)$ with $1$ in the $i$th entry. The zero ring is the only one which is normal $\P$. Therefore, Corollary \ref{proda} tells us that $R=R_1 \times \cdots \times R_n$ is a mis\`{e}re $\P$-position if and only if $R_j=0$ for all $j$ except for one index $i$, and $R_i \cong R$ is a mis\`{e}re $\P$-position.
\end{ex}
 
\begin{cor} \label{ringprod2}
Let $R$ be a commutative ring which is a mis\`{e}re $\P$-position. Then $R$ cannot be written as a product of two non-trivial rings. In other words, $R$ does not contain any non-trivial idempotent elements.
\end{cor}
 
\begin{rem}
Commutative rings with the property in Corollary \ref{ringprod2} also called connected because their prime spectrum $\Spec(R)$ is a connected topological space \cite[Chapter 1, Exercise 22]{AM69}. We may also state the result positively as follows: If $R = R_1 \times R_2$ is a product of two non-trivial commutative rings, then $R$ is mis\`{e}re $\N$. However, the proof in Proposition \ref{sel-misere} does only produce a winning move if the game outcome of $R_1$ (or $R_2$) was already known: If $R_1$ is mis\`{e}re $\N$, then there is some $0 \neq x \in R_1$ such that $R_1 / \langle x \rangle$ is mis\`{e}re $\P$. Then $R/\langle (x,1) \rangle \cong R_1/\langle x \rangle$ is mis\`{e}re $\P$. If $R_1$ is mis\`{e}re $\P$, then $R/\langle (0,1) \rangle \cong R_1$ is mis\`{e}re $\P$.
\end{rem}
  

\section{The game of abelian groups} \label{sec:ab}

In this section we analyze the game of abelian groups. We have already seen that the ending condition is satisfied precisely for finitely generated abelian groups. Their structure is well-known \cite[Chapter I, \S 8]{L}.
  
\begin{thm}[Structure theorem] \label{structuretheorem}
Let $A$ be a finitely generated abelian group.
\begin{enumerate}[leftmargin=4ex]
\item If $A$ is finite, then there are unique natural numbers $s \geq 0$ and $n_1,\dotsc,n_s > 1$ satisfying $n_i  \mid  n_{i+1}$ for $1 \leq i < s$ such that $A \cong \Z/n_1 \oplus \cdots \oplus \Z/n_s$. Here, $s$ is the smallest natural number such that $A$ can be generated by $s$ elements.
\item If $A$ is finite, then $A = \bigoplus_{p \text{ prime}} A_p$, where $A_p  \coloneqq  \bigcup_{n \geq 0} \ker(p^n : A \to A)$ is the $p$-Sylow subgroup of $A$.
\item In the general case, the torsion subgroup $A_t \coloneqq \bigcup_{n>0} \ker(n : A \to A)$ is finite and there is a unique natural number $r \geq 0$, the rank of $A$, such that $A \cong A_t \oplus \Z^r$.
\end{enumerate}
\end{thm}

There is also a version of the structure theorem with prime powers, but this means that we have much more factors in the direct products, and hence the winning moves will be longer to write down. This is why we have decided to use divisor sequences. The structure theorem or rather a refinement of it will enable us find a beautiful characterization of the $\P$-positions. Our analysis also works for finitely generated $R$-modules, where $R$ is a principal ideal domain, because the structure theorem also holds for them \cite[Chapter III, \S 7]{L}. For simplicity of exposition, we will restrict to the case $R=\Z$ here.
 
\subsection{Finite abelian groups}

\begin{prop} \label{redp}
If $A,B$ are finite abelian groups of coprime orders, then the game of $A \times B$ is the selective compound of the games of $A$ and $B$. In particular, if $A$ is a finite abelian group, then the game of $A$ is the selective compound of the games of the $p$-Sylow subgroups $A_p$.
\end{prop}

\begin{proof} It is enough to verify the conditions of Proposition \ref{proda}, i.e.\ that for every pair $A,B$ as in the claim the canonical homomorphism
\[(A \times B)/\langle (a,b) \rangle \to A/\langle a \rangle \times B / \langle b \rangle\]
is an isomorphism for all $a \in A$ and $b \in B$. This is equivalent to $\langle (a,b) \rangle = \langle a \rangle \times \langle b \rangle$. Since $\subseteq$ is obvious, it suffices to check that the order of $(a,b)$ is the product of the orders of $a$ and $b$. In general, the order of $(a,b)$ is the least common multiple of the orders of $a$ and $b$. Since they are coprime, the result follows.
 \end{proof}
 
Thus, we may restrict to abelian $p$-groups. However, some aspects of the game are better formulated without this restriction. So let us stay with arbitrary finite abelian groups for the moment. The first step is to characterize all options of the game. This characterization will show that the game of finite abelian groups is actually a purely number-theoretic game. The proof is laborious, but the rest will be rather formal. In the following, we will make the common abuse of notation to denote the image of an integer $m \in \Z$ in a quotient group $\Z/n$ also by $m$. It will become handy to describe abelian groups by generators and relations \cite[Chapter I, \S 12]{L}.
 
\begin{prop} \label{compute}
Let $A$ be a finite abelian group, say $A \cong \Z/n_1 \oplus \cdots \oplus \Z/n_s$ with $n_i  \mid  n_{i+1}$ and $n_i \geq 1$. Then a finite abelian group $B$ is isomorphic to $A/\langle x \rangle$ for some element $x \in A$ if and only if there is a sequence of natural numbers $m_1,\dotsc,m_s$ satisfying
\[m_1  \mid  n_1  \mid  m_2  \mid  n_2  \mid  \cdots  \mid  m_s  \mid  n_s\]
and
\[B \cong \Z/m_1 \oplus \cdots \oplus \Z/m_s.\]
If $m_1,\dotsc,m_s$ is such a sequence, then we may choose
\[x = m_1 \oplus m_1 \cdot \frac{m_2}{n_1} \oplus \cdots \oplus m_1 \cdot \frac{m_2}{n_1} \cdot \,\cdots\, \cdot \frac{m_s}{n_{s-1}}.\]
\end{prop}

We note that the following proof is more or less equivalent to the structure theorem for finite abelian groups.

\begin{proof}
Let us first verify the easy direction. For $x$ defined as above, we want to show $A / \langle x \rangle \cong \Z/m_1 \oplus \cdots \oplus \Z/m_s$. Let us make an induction on $s$, the cases $s=0$ and $s=1$ being trivial. The quotient $A/\langle x \rangle$ is given by (commuting) generators $e_1,\dotsc,e_s$ and relations $n_i e_i = 0$ for $1 \leq i \leq s$ as well as the relation
\[m_1 e_1 + m_1 \cdot \frac{m_2}{n_1} e_2 + \cdots = 0.\]
This can also be written as $m_1 e'_1 = 0$, where
\[e'_1 = e_1 + \frac{m_2}{n_1} e_2 + \frac{m_2}{n_1} \cdot \frac{m_3}{n_2} + \cdots.\]
We find a new presentation with the generator $e_1$ replaced by $e'_1$, and the relation $n_1 e_1 = 0$ replaced by
\[n_1 e'_1 = m_2 e_2 + m_2 \cdot \frac{m_3}{n_2} + \cdots.\]
The left hand side vanishes because of $m_1 e'_1 = 0$ and $m_1  \mid  n_1$. Hence, the relation does not contain $e'_1$ anymore and we can split off $\langle e'_1 : m_1 e'_1 = 0\rangle \cong \Z/m_1$, the rest being isomorphic to $\Z/m_2 \oplus \cdots \oplus \Z/m_s$ by the induction hypothesis. Thus, we obtain $\Z/m_1 \oplus \Z/m_2 \oplus \cdots \oplus \Z/m_s$.

Now for the other direction, we assume that $x \in A$ is an arbitrary element. We claim that there are natural numbers $m_1  \mid  n_1  \mid  m_2  \mid  n_2  \mid  \cdots$ such that $A/\langle x \rangle$ is isomorphic to $\Z/m_1 \oplus \cdots \oplus \Z/m_s$. This will be done by induction on $s$. By Proposition \ref{redp} we may assume that everything is a power of a prime $p$. Technically, this is not an important ingredient for the proof, but it simplifies the complicated relation $ \mid $ to the simple relation  $\leq$. Write $n_i = p^{k_i}$ with $k_i \geq 0$. Then we claim that there are natural numbers $m_i \geq 0$ such that $m_1 \leq k_1 \leq m_2 \leq k_2 \leq m_3 \leq \cdots$ such that $A/\langle x \rangle$ is isomorphic to $\Z/p^{m_1} \oplus \cdots \oplus \Z/p^{m_s}$. Now consider $x_i \in \Z/p^{k_i}$ and lift it to some natural number, also denoted by $x_i$. We may write $x_i = p^{r_i} u_i$ for some unique $0 \leq r_i \leq k_i$ and $u_i$ with $p \nmid u_i$. Since multiplication with $u_i$ induces an automorphism of $\Z/p^{k_i}$, we may even assume that $x_i=p^{r_i}$.

Next, we give a recursive description of the quotient
\[A_{k,r}  \coloneqq  (\Z/p^{k_1} \oplus \cdots \oplus \Z/p^{k_s}) / \langle (p^{r_1},\dotsc,p^{r_s})\rangle.\]
This can also be written as the abelian group defined by generators $e_1,\dotsc,e_s$, relations $p^{k_i} e_i=0$ for $1 \leq i \leq s$, as well as the relation
\[p^{r_1} e_1 + \cdots + p^{r_s} e_s = 0.\]
Choose $1 \leq l \leq s$ in such a way that $r_l$ becomes minimal. If we replace $e_l$ by the new generator
\[e'_l  \coloneqq  \sum_{i} p^{r_i - r_l} e_i = e_l + \sum_{i \neq l} p^{r_i - r_l} e_i,\]
the above relation becomes $p^{r_l} e'_l = 0$. In terms of $e'_l$, the relation $p^{k_l} e_l = 0$ becomes
\[p^{k_l} e'_l = \sum_{i \neq l} p^{k_l + r_i - r_l} e_i.\]
The left hand side vanishes because of $p^{r_l} e'_l = 0$ and $r_l \leq k_l$. Thus, we can split off $\langle e'_l \rangle \cong \Z/p^{r_l}$. Also, since $p^{k_i} e_i = 0$, we could equally well replace the coefficient of $e_i$ in the sum above by $p^{r'_i}$, where
\[r'_i  \coloneqq  \min(k_l + r_i - r_l,k_i).\]

For $i<l$ we have $r'_i = k_i$, so that we may split off $\langle e_i \rangle \cong \Z/p^{k_i}$. Thus, if we define $k'_i = k_i$ for $i>l$, we obtain the recursive expression
\[A_{k,r} \cong \Z/p^{r_l} \oplus \Z/p^{k_1} \oplus \cdots \oplus \Z/p^{k_{l-1}} \oplus A_{k',r'}.\]
Let us add to the induction hypothesis that $r_l$ is the smallest exponent in the decomposition, i.e.\ $r_l = m_1$. Applying the induction hypothesis to $A_{k',r'}$ we get numbers $m_{l+1} \leq k_{l+1} \leq m_{l+2} \leq \cdots \leq k_s$ such that $A_{k',r'} \cong \Z/p^{m_{l+1}} \oplus \cdots \oplus \Z/p^{m_s}$. Besides, $m_{l+1}$ is the minimum of the $r'_i$, which is $\geq k_l$. Now let us define $m_1 = r_l$ and $m_i = k_{i-1}$ for $1 < i \leq l$. Then $A_{k,r} \cong \Z/p^{m_1} \oplus \cdots \oplus \Z/p^{m_s}$ and we have
\[m_1 \leq k_1 = m_2 \leq k_2 = m_3 \leq \cdots \leq k_{l-1} = m_l \leq k_l \leq m_{l+1} \leq k_{l+1} \leq \cdots \leq k_s,\]
as required.
\end{proof}
  
\begin{prop} \label{equiv}
The game of finite abelian groups is equivalent to the following number-theoretic game: The positions are divisor sequences $n_1  \mid  \cdots  \mid  n_s$ of natural numbers $\geq 1$, where we identify $1 \mid 1 \mid \cdots \mid n_1 \mid \cdots \mid n_s$ with $n_1 \mid \cdots \mid n_s $. There is a move from $n_1  \mid  \cdots  \mid  n_s$ to $m_1  \mid  \cdots  \mid  m_s$ if and only if $m_1  \mid  n_1  \mid  m_2  \mid  n_2  \mid  \cdots  \mid  m_s  \mid  n_s$ and for at least one $1 \leq i \leq s$ we have $m_i < n_i$. The only terminal position is the empty sequence.
\end{prop}

\begin{proof}
This follows from Proposition \ref{compute}. In fact, $n_1  \mid  \cdots  \mid  n_s$ corresponds to the group $\Z/n_1 \oplus \cdots \oplus \Z/n_s$.
\end{proof}
 
\begin{rem}
We conjecture that the game of finitely generated abelian groups is equivalent to the number-theoretic game of divisor sequences $n_1 \mid \cdots \mid n_s$ of natural numbers $\geq 0$, where the zeroes at the end correspond to direct summands of the form $\Z/0 = \Z$ of the abelian group. Several results in the next sections support this conjecture. It would follow from an appropriate generalization of Proposition \ref{compute}.
\end{rem}
 
Now we can easily determine the $\P$-positions:
 
\begin{prop} \label{P}
In the number-theoretic game described in Proposition \ref{equiv}, the divisor sequence $n_1  \mid  \cdots  \mid  n_s$ is a normal $\P$-position if and only if it is a square in the following sense: Either $s$ is even and $n_1=n_2$, $n_3=n_4$, $\dotsc$, $n_{s-1}=n_s$, or $s$ is odd and $n_1=1$, $n_2=n_3$, $\dotsc$, $n_{s-1}=n_s$.
\end{prop}

\begin{proof}
Clearly the terminal position, which is $\P$, is a square with $s=0$. We have to prove that every non-square moves to some square, and that a square cannot move to another square.

Assume that a square $n_1  \mid  \cdots  \mid  n_s$ moves to some square $m_1  \mid  \cdots  \mid  m_s$. We may assume that $s$ is even; otherwise add $1$ on the left. For even $i \geq 2$ we have $n_i = n_{i-1}  \mid  m_i  \mid  n_i$, thus $m_i = n_i$. Since both sequences are squares, this already implies $m_i = n_i$ for all $i$. This is a contradiction.

Assume that $n_1  \mid  \cdots  \mid  n_s$ is not a square. If $s$ is even, define $m_i  \coloneqq  m_{i+1}  \coloneqq  n_i$ for all odd $i$. Then we have $m_1 = n_1 = m_2  \mid  n_2  \mid  m_3 = n_3 = m_4  \mid  \cdots$, and $m$ is a square. In particular, $m \neq n$. Hence, $m$ is a winning move. The case that $s$ is odd can be reduced to this case by adding $1$ on the left. The winning move is here $m_1  \coloneqq  1$ and $m_i  \coloneqq  m_{i+1}  \coloneqq  n_i$ for all even $i>1$.
 \end{proof}
 
Now we can prove the main theorems about the game of finite abelian groups.

\begin{thm} \label{fin}
Let $A$ be a finite abelian group.
\begin{enumerate}[leftmargin=4ex]
\item $A$ is a normal $\P$-position if and only if $A$ is a square, i.e.\ $A \cong B^2$ for some finite abelian group $B$.
\item If $A = \Z/n_1 \oplus \cdots \oplus \Z/n_s$ with $n_i  \mid  n_{i+1}$ is not a square, then a winning move is
\[x = 0 \oplus n_1 \oplus \dfrac{n_1 \cdot n_3}{n_2} \oplus \dfrac{n_1 \cdot n_3}{n_2} \oplus \cdots \oplus \dfrac{n_1 \cdot n_3 \cdot \,\cdots\, \cdot n_{s-1}}{n_2 \cdot \,\cdots\, \cdot n_{s-2}} \oplus \dfrac{n_1 \cdot n_3 \cdot  \,\cdots\, \cdot n_{s-1}}{n_2 \cdot \,\cdots\, \cdot n_{s-2}}\]
if $s$ is even, and
\[x =  1 \oplus  \frac{n_2}{n_1} \oplus \frac{n_2}{n_1} \oplus \frac{n_2 \cdot n_4}{n_1 \cdot n_3} \oplus \cdots \oplus \frac{n_2 \cdot n_4 \cdot \,\cdots\, \cdot n_{s-1}}{n_1 \cdot n_3 \cdot \,\cdots\, \cdot n_{s-2}} \oplus \frac{n_2 \cdot n_4 \cdot \,\cdots\, \cdot n_{s-1}}{n_1 \cdot n_3 \cdot \,\cdots\, \cdot n_{s-2}}\]
if $s$ is odd. In that case, we have
\[A/\langle x \rangle \cong \left\{\begin{array}{ll} (\Z/n_1 \oplus \Z/n_3 \oplus \cdots \oplus \Z/n_{s-1})^2 & \text{ if } s \text{ is even,} \\ (\Z/n_2 \oplus \Z/n_4 \oplus \cdots \oplus \Z/n_{s-1})^2 & \text{ if } s \text{ is odd.} \end{array}\right.\]
\end{enumerate}
\end{thm}

\begin{proof}
1.\ follows from Propositions \ref{equiv} and \ref{P}, and  2.\ follows from an inspection of the proofs of Propositions \ref{P} and \ref{compute}.
 \end{proof}
 
\begin{ex}
For example, $\Z/4 \oplus \Z/8 \oplus \Z/40$ is a normal $\N$-position. Player I quotients out $1 \oplus 2 \oplus 2$, since $8/4=2$. The quotient is isomorphic to the square $\Z/8 \oplus \Z/8$. Player II has many choices, but he loses in any case. Let us demonstrate this for the element $4 \oplus 0$. Then Player I gets $\Z/4 \oplus \Z/8$ and of course he quotients out $0 \oplus 4$, because this gives $\Z/4 \oplus \Z/4$ for Player II. If he wants to postpone his inevitable defeat, he could try $2 \oplus 2$ with quotient $\cong \Z/2 \oplus \Z/4$. The next moves are $\Z/2 \oplus \Z/2$ by Player I, $\Z/2$ by Player II and finally $0$ by Player I, who wins.
\end{ex}
  
\begin{thm} \label{finmis} 
Let $A$ be a finite abelian group. Then $A$ is a mis\`{e}re $\P$-position if and only if $A$ is
\begin{itemize}[leftmargin=4ex]
\item either elementary abelian of odd dimension,
\item or a square, without being elementary abelian.
\end{itemize}
\end{thm}

Thus, the only difference to the normal $\P$-positions are the elementary abelian groups $(\Z/p)^s$, which become mis\`{e}re $\P$ if and only if $s$ is odd.

\begin{proof}
According to Proposition \ref{redp} and Corollary \ref{proda}, it suffices to treat the case that $A$ is a finite abelian $p$-group, say $A = \Z/p^{k_1} \oplus \cdots \oplus \Z/p^{k_s}$ with $k_1 \leq \cdots \leq k_s$.
  
We say that $A$ is $\P'$ if it is either elementary abelian of odd dimension, or it is a square, without being elementary abelian. We have to show the three properties characterizing mis\`{e}re $\P$-positions (Proposition \ref{charNP}). The terminal positions are elementary abelian of dimension $0$, thus not $\P'$. Next, we have to show that if $A \neq 0$ is not $\P'$, then some option of $A$ is $\P'$. If $A$ is elementary abelian, then its dimension is even $\neq 0$, and in fact every move reduces the dimension by one, so that we end up with something which is $\P'$. If $A$ is not elementary abelian, then it is not a square. By Theorem \ref{fin}, there is some $0 \neq x \in A$ such that $A/\langle x \rangle$ is a square, namely isomorphic to $(\Z/p^{k_1} \oplus \cdots \oplus \Z/p^{k_{s-1}})^2$ if $s$ is even, and otherwise to $(\Z/p^{k_2} \oplus \cdots \oplus \Z/p^{k_{s-1}})^2$. If these are not elementary abelian, they are $\P'$ we are done. Now we assume that they are elementary abelian, i.e.\ $k_{s-1}=1$. Thus, $A = (\Z/p)^{s-1} \oplus \Z/p^{k_s}$. We have $k_s > 1$. If $s$ is even, the winning move is now $0 \oplus \cdots \oplus 0 \oplus 1$, since the quotient is $(\Z/p)^{s-1}$, which is elementary abelian of odd dimension and therefore $\P'$. If $s$ is odd, the winning move is $0 \oplus \cdots \oplus 0 \oplus p$, since the quotient is $(\Z/p)^s$, therefore also $\P'$.

Finally, we have to show that if $A$ is $\P'$, then for every $0 \neq x \in A$ the abelian group $A'=A/\langle x \rangle$ is not contained in $\P'$. This is clear if $A$ is elementary abelian. Otherwise, $A$ is a square, $s$ is even, and $A'$ cannot be a square by Theorem \ref{fin}. For a contradiction, we assume that $A'$ is $\P$. Then $A'$ is elementary abelian of odd dimension. Since $pA'=0$, we have $pA \subseteq \langle x \rangle$. Thus, $pA$ is cyclic. On the other hand, it contains $p (\Z/p^{k_{s-1}} \oplus \Z/p^{k_s}) \cong (\Z/p^{k_{s}-1})^2$, which is only cyclic when $k_s=1$. But this implies $k_i=1$ for all $i$, i.e.\ $A$ is elementary abelian. This contradiction finishes the proof.
\end{proof}
 
\begin{ex}
For example, $\Z/2 \oplus \Z/6 \oplus \Z/6$ is a mis\`{e}re $\N$-position. We may also represent this group as $(\Z/2)^3 \oplus (\Z/3)^2$. Here are two possible sequences of moves:
\[(\Z/2)^3 \oplus (\Z/3)^2 \I  (\Z/2)^2 \oplus (\Z/3)^2 \II \Z/2 \oplus (\Z/3)^2 \I \Z/3 \II 0,\]
\[(\Z/2)^3 \oplus (\Z/3)^2 \I  (\Z/2)^3 \oplus \Z/3 \II (\Z/2)^3 \oplus \Z/3 \I (\Z/2)^3  \II (\Z/2)^2 \I \Z/2  \II 0.\]
\end{ex}
 
\subsection{Finitely generated abelian groups}

The classification of $\P$-positions may be generalized from finite abelian groups to finitely generated abelian groups as follows.

\begin{thm} \label{fg}
Let $A$ be a finitely generated abelian group. Then $A$ is a normal $\P$-position if and only if $A$ is a square, i.e.\ $A \cong B^2$ for some finitely generated abelian group $B$.
\end{thm}

\begin{proof} In the finite case, we may use Theorem \ref{fin}. In the general case, we may write $A \cong A_t \oplus \Z^r$, where $A_t$ is the finite torsion subgroup of $A$ and $r \geq 0$ is the rank of $A$. It is easy to see that $A$ is a square if and only if $A_t$ is a square and $r$ is even. As before, it is enough to prove that every non-square moves to some square and that every square cannot move to another square.

Assume that $A$ is not a square. If $r$ is even, then $A_t$ is not a square and by the finite case there is some $0 \neq x \in A_t$ such that $A_t / \langle x \rangle$ is a square. But then
\[A / \langle (x \oplus 0) \rangle \cong A_t / \langle x \rangle \oplus \Z^r\]
is a square. If $r$ is odd, it is enough to consider the case $r=1$ by ignoring the direct summand $\Z^{r-1}$ which is already a square. If $A_t$ is generated by $s$ elements, say $A_t \cong \Z/n_1 \oplus \cdots \oplus \Z/n_s$, let $n_{s+1}  \coloneqq  0$ and apply the winning strategy of Theorem \ref{fin} to $A \cong \Z/n_1 \oplus \cdots \oplus \Z/n_s \oplus \Z/n_{s+1}$. This works since we never divided through the last number $n_{s+1}$; in fact we didn't use it at all. Thus, if $s$ is even, there is a move from $A$ to the square $(\Z/n_2 \oplus \Z/n_4 \oplus \cdots \oplus \Z/n_s)^2$. If $s$ is odd, there is a move to the square $(\Z/n_1 \oplus \Z/n_3 \oplus \cdots \oplus \Z/n_s)^2$.

Now we assume that $A$ is a square of rank $r$ and there is some move to a square $B$. In other words, there is some cyclic subgroup $C \neq 0$ of $A$ such that $A/C \cong B$. When $C$ is finite, we have $C \subseteq A_t$ and therefore $B \cong \Z^r \oplus A_t/C$. Since $B$ is a square, it follows that $A_t/C$ is a square, which is impossible by the finite case since also $A_t$ is a square. Now we assume that $C$ is infinite. Then $B$ is of rank $r-1$, which is odd, a contradiction.
\end{proof}
 
\begin{thm} \label{fgmis}
Let $A$ be a finitely generated abelian group. Then $A$ is a mis\`{e}re $\P$-position if and only if $A$ is
\begin{itemize}[leftmargin=4ex]
\item either finite elementary abelian of odd dimension,
\item or a square, but not finite elementary abelian
\end{itemize}
In particular, if $A$ is infinite and a square, then $A$ is mis\`{e}re $\P$.
\end{thm}

\begin{proof}
Let $\P'$ be the class of groups described in the theorem. Clearly $0 \notin \P'$. Again we have to verify that $A \in \P'$ cannot move to some $B \in \P'$, and that every $0 \neq A \notin \P'$ moves to some $B \in \P'$. If $A$ is finite, both follow from Theorem \ref{finmis}. Now we assume that $A$ is infinite.

If $A \in \P'$, then $A$ is a square, and for every move $B \coloneqq A/\langle x \rangle$ it follows from Theorem \ref{fg} that $B$ is not a square. If $B \in \P'$, it would follow that $B$ is finite, in fact elementary abelian of odd dimension and therefore of rank $0$. It follows that $1 \leq \rank(A) = \rank(\langle x \rangle) \leq 1$, thus $\rank(A)=1$. But this contradicts $A$ being a square. Thus, $B \notin \P'$.

If $0 \neq A \notin \P'$, then $A$ is not a square, and by Theorem \ref{fg} there is some $0 \neq x \in A$ such that $B \coloneqq A/\langle x \rangle$ is a square. If $B \in \P'$, we would be done. Otherwise, $B$ is finite and elementary abelian of even dimension. It follows once again $\rank(A)=1$ and we may write $A = \Z/n_1 \oplus \cdots \oplus \Z/n_s \oplus \Z$ for some $n_1  \mid  \cdots  \mid  n_s$ with $n_i > 1$. The proof of Theorem \ref{fg} shows that we can choose $x$ in such a way that $B \cong (\Z/n_s \oplus \Z/n_{s-2} \oplus \cdots)^2$. Since $B$ is elementary abelian, it follows that $n_s$ is some prime number $p$. But then we even have $n_1=\cdots=n_s=p$, i.e.\ $A = (\Z/p)^s \oplus \Z$. Now, if $s$ is odd, we quotient out $0^{\oplus s} \oplus 1$ to obtain $(\Z/p)^s$, which is $\P'$. If $s$ is even, we quotient out $0^{\oplus s} \oplus p$ to obtain $(\Z/p)^{s+1}$, which is again $\P'$. This finishes the proof.
 \end{proof}

\begin{rem}
The same analysis works for the game of $R$-modules, where $R$ is a principal ideal domain, because the structure theorem also holds for them \cite[Chapter III, \S 7]{L}. Namely, a finitely generated $R$-module $M$ is normal $\P$ if and only if it is a square. If $R$ is not a field, then $M$ is mis\`{e}re $\P$ if and only if it is either a vector space over some $R/p$ of odd dimension (where $p \in R$ is some prime element), or it is a square, but not a vector space over any $R/p$.
\end{rem}
 
\begin{ex} \label{Bezout}
Theorem \ref{fg} predicts that the abelian group $\Z \oplus  \Z$ is a normal $\P$-position. Let us verify this directly and thereby make the game more explicit. Pick any non-trivial element $(n,m) \in \Z \oplus  \Z$. By B\'{e}zout's theorem, there are integers $p,q \in \Z$ such that $pn + qm = \gcd(n,m)$. Then, the $2 \times 2$-matrix
\[\begin{pmatrix}
q  & \frac{n}{\gcd(n,m)}  \\ -p & \frac{m}{\gcd(n,m)}
\end{pmatrix}\]
is invertible with inverse
\[\begin{pmatrix}
\frac{m}{\gcd(n,m)}  & -\frac{n}{\gcd(n,m)}  \\ p & q
\end{pmatrix}.\]
Hence, the transformation
\[x' = q x - p y,\quad y' = \tfrac{n}{\gcd(n,m)} x + \tfrac{m}{\gcd(n,m)} y\]
yields an isomorphism
\begin{align*}
(\Z \oplus  \Z)/\langle (n,m) \rangle &= \langle x,y : nx + my = 0\rangle \cong \langle x',y' : \gcd(n,m) \cdot y' = 0\rangle\\
& \cong \Z/\gcd(n,m) \oplus \Z.
\end{align*}
From this group, the winning move is to quotient out $(0,\gcd(n,m))$, because this results in the square $(\Z/ \gcd(n,m))^2$, which is a $\P$-position by the finite case. For example, a possible sequence of moves is the following:
\[\Z \oplus \Z \I (\Z \oplus \Z)/\langle (2,4) \rangle \cong \Z/2 \oplus \Z \II \Z/2 \oplus \Z/2 \I \Z/2 \II 0.\]
\end{ex}

\subsection{Computation of some nimbers} \label{subsec:nimb}

If $A$ is a finitely generated abelian group, then the game of $A$ is determined by the nimber $\alpha(A)$ (see Remark \ref{SG}). This ordinal number is defined recursively by
\[\alpha(A) = \mex \{\alpha(A/\langle a \rangle) : 0 \neq a \in A \}.\]
We have $\alpha(A)=0$ if and only if $A$ is a $\P$-position (under the normal play rule), i.e.\ $A$ is a square (Theorem \ref{fg}). The nimber carries much more information than just the knowledge about which player wins. Accordingly it is more difficult to compute. An induction shows that the nimber of a finitely generated abelian group $A$ is a countable ordinal number, which is finite if $A$ is finite.

First, we will compute the nimbers of cyclic groups. If $0 \neq n \in \Z$, let us write $\Omega(n)$ for the number of prime divisors of $n$ counted with multiplicity.
 
\begin{lem} \label{Zn}
For $0 \neq n \in \Z$ we have $\alpha(\Z/n) = \Omega(n)$.
\end{lem}

\begin{proof}
By induction on $n$, we have that $\alpha(\Z/n)$ is the mex of the numbers $\Omega(m)$, where $m$ is a proper divisor of $n$. In that case we have $\Omega(m)<\Omega(n)$. Moreover, every natural number $<\Omega(n)$ has this form.
\end{proof}

\begin{cor} \label{Z}
We have $\alpha(\Z) = \omega$.
\end{cor}

\begin{proof}
We have $\alpha(\Z/2^n)=n$ for all $n \in \IN$  (Lemma \ref{Zn}). If $0 \neq n \in \Z$, then $\alpha(\Z/n) < \omega$ since $\Z/n$ is finite; alternatively, we may use Lemma \ref{Zn} again. Hence, $\alpha(\Z) = \mex(\omega)=\omega$.
\end{proof} 
 
Our next goal is to compute the nimbers of $2$-generated abelian $p$-groups, where $p$ is a fixed prime number. Every such group is isomorphic to $\Z/p^n \oplus \Z/p^m$ for uniquely determined natural numbers $n \leq m$. We abbreviate $\alpha(\Z/p^n \oplus \Z/p^m)$ by $\alpha(n,m)$. By Proposition \ref{compute} the options of the group $\Z/p^n \oplus \Z/p^m$ are those groups $\Z/p^{n'} \oplus \Z/p^{m'}$ for which $n' \leq n \leq m' \leq m$ and $(n,m) \neq (n',m')$ hold. It follows that
\[\alpha(n,m) = \mex \{\alpha(n',m') : n' \leq n \leq m' \leq m, ~ (n,m) \neq (n',m')\}.\]
This enables us to compute some values recursively, see Figure \ref{tabelle}.
\begin{figure}[t]
\begin{center}
\begin{tabular}{r|*{15}{P{0.98em}}}
 $n \backslash m$\!\! & 0 & 1 & 2 &  3  &  4  &  5  &  6  &  7  & 8 &  9 &  10  &  11  &  12  &  13  & 14 \\[0.1em] \hline
 0  &  0 &  1 &  2 &  3 &  4 &  5 &  6 &  7 &  8 &  9 & 10 & 11 & 12 & 13 & 14  \\[0.1em]
 1  &   &  \textbf{0} &  3 &  4 &  5 &  6 &  7 &  8 &  9 & 10 & 11 & 12 & 13 & 14  & 15 \\[0.1em]
 2  &  	 &    &  \textbf{0} &  \textbf{1} &  6 &  7 &  8 &  9 & 10 & 11 & 12 & 13 & 14 & 15 & 16 \\[0.1em]
 3  &	   &    &    &  \textbf{0} &  \textbf{2} &  8 &  9 & 10 & 11 & 12 & 13 & 14 & 15 & 16 & 17 \\[0.1em]
 4  &	   &    &    &    &  \textbf{0} &  \textbf{1} &  \textbf{3} & 11 & 12 & 13 & 14 & 15 & 16 & 17 & 18 \\[0.1em]
 5  &	   &    &    &    &    &  \textbf{0} &  \textbf{2} &  \textbf{4} & 13 & 14 & 15 & 16 & 17 & 18 & 19 \\[0.1em]
 6  &	   &    &    &    &    &    &  \textbf{0} &  \textbf{1} &  \textbf{5} & 15 & 16 & 17 & 18 & 19 & 20 \\[0.1em]
 7	&	   &    &    &    &    &    &    &  \textbf{0} &  \textbf{2} &  \textbf{3} &  \textbf{6} & 18 & 19 & 20 & 21 \\[0.1em]
 8	&	   &    &    &    &    &    &    &    &  \textbf{0} &  \textbf{1} &  \textbf{4} &  \textbf{7} & 20 & 21 & 22 \\[0.1em]
 9	&	   &    &    &    &    &    &    &    &    &  \textbf{0} &  \textbf{2} &  \textbf{5} &  \textbf{8} & 22 & 23 \\[0.1em]
 10	&	   &    &    &    &    &    &    &    &    &    &  \textbf{0} &  \textbf{1} &  \textbf{3} &  \textbf{9} & 24 \\[0.1em]
 11 &	   &    &    &    &    &    &    &    &    &    &    &  \textbf{0} &  \textbf{2} &  \textbf{4} & \textbf{6} \\
\end{tabular}
\caption{The values of $\alpha(n,m)$ for $0 \leq n \leq 11$ and $n \leq m \leq 14$}\label{tabelle}\vspace{-1em}
\end{center} 
\end{figure}

For example, $\alpha(4,8)=12$ because the block with corners $4,8,0,?$ contains all numbers $0,1,\dotsc,11$. For another example, $\alpha(6,8)=5$ because the block with corners $6,8,0,?$ contains the numbers $0,1,2,3,4$ and $6,7,\dotsc,13$. The values in Figure \ref{tabelle} indicate the following pattern:
\begin{itemize}[leftmargin=4ex,noitemsep]
\item For fixed $n$, we have $\alpha(n,m)=n+m$ for large $m$.
\item For fixed $k$, the diagonal $(\alpha(n,n+k))_{n \geq 0}$ eventually becomes periodic (printed in boldface) with period length $k+1$.
\item More precisely, the period is given by $\Delta_k,\Delta_k+1,\dotsc\Delta_k + k = \Delta_{k+1}-1$, where $\Delta_k = \frac{1}{2} k(k+1)=1+2+\cdots+k$ is the triangular number.
\item This period starts when $n > \Delta_k$.
\end{itemize}

Let us verify this pattern. We denote by $(a \bmod k+1)$ the unique natural number $0 \leq r \leq k$ such that $a \equiv r \bmod k+1$.

\begin{thm} \label{nimbtwo}
For natural numbers $n \leq m$ with $k \coloneqq m-n$, the nimber of the abelian group $\Z/p^n \oplus \Z/p^m$ equals
\[\alpha(n,m) = \left\{\begin{array}{cc} n+m &\text{if }n \leq \Delta_k, \\ \Delta_k + (n-\Delta_k-1 \bmod k+1) &\text{if } n > \Delta_k. \end{array}\right.\]
\end{thm}

\begin{proof}
We assume that the claim is true for all $n' \leq n \leq m' \leq m$ with $(n',m') \neq (n,m)$ and prove it for $(n,m)$. The reader may find it helpful to visualize the proof using Figure \ref{tabelle}.

\textsc{Case A.} We assume $n \leq \Delta_k$. We claim that $\{\alpha(n',m') : \dotsc\}$ is the set of natural numbers $<n+m$, so that its mex is $\alpha(n,m)=n+m$.

\textsc{1.\ Step.} We prove that every number $<n+m$ arises as $\alpha(n',m')$. In fact, for $n' < n$ we have $\alpha(n',m)=n'+m$ (since $n' \leq n \leq \Delta_{m-n} \leq \Delta_{m-n'}$). Hence, the numbers $m,1+m,\dotsc,(n-1)+m$ occur. For $n \leq m' < m$ we have $\alpha(0,m') = m'$ (because of $0 \leq \Delta_{m'}$). Hence, the numbers $n,\dotsc,m-1$ occur. Finally, let $0 \leq \ell < n$. Choose $k' < k$ such that $\Delta_{k'} \leq \ell < \Delta_{k'+1}$. Write $n-(\ell+1) = q(k'+1)+r$ with $q \geq 0$ and $0 \leq r \leq k'$. Let $n'  \coloneqq  n-r$ and $m'  \coloneqq  n'+k'$. Clearly, $n' \leq n$ and $n'=q(k'+1)+(\ell+1) \equiv \ell+1 \bmod k'+1$. We have $n \leq m'$ since this is equivalent to $r \leq k'$. We have $m' < m$ since this is equivalent to $k' < k+r$. Finally, notice that $n' > \ell \geq \Delta_{k'}$. Therefore, we arrive at
\[\alpha(n',m') = \Delta_{k'} + (n' - \Delta_{k'}-1 \bmod k'+1) = \Delta_{k'} + (\ell-\Delta_{k'} \bmod k'+1) = \ell.\]

\textsc{2.\ Step.} We prove that $\alpha(n',m') < n+m$ for each option $(n',m')$. Let $k' = m'-n'$. If $n' \leq \Delta_{k'}$, we have $\alpha(n',m') = n'+m' \leq n'+m \leq n + m$, with no equality since otherwise $(n',m') = (n,m)$. Hence, $\alpha(n',m') < n+m$. Now let us assume $n' > \Delta_{k'}$. Then $\alpha(n',m') \in [\Delta_{k'},\Delta_{k'}+k']$, so that
\[\alpha(n',m') \leq \Delta_{k'}+k' < n'+k' = m' \leq m \leq n+m.\]

\textsc{Case B.} We assume $n > \Delta_k$. Let us write $n-\Delta_k-1 = q(k+1) + r$ with $q \geq 0$ and $0 \leq r \leq k$. We want to prove $\alpha(n,m) = \Delta_k+r$. For this, we have to prove that $\{\alpha(n',m') : \dotsc\}$ contains all numbers $<\Delta_k+r$, but not $\Delta_k+r$. Notice that, in contrast to Case A, numbers $>\Delta_k+r$ \emph{do} occur in that set.

\textsc{1.\ Step.} We prove that every number $\ell <\Delta_k+r$ arises as $\alpha(n',m')$. First, let us assume $\Delta_k \leq \ell < \Delta_k+r$, i.e.\ we are in the same diagonal. Write $\ell = \Delta_k + r'$ with $0 \leq r' < r$. Let $\delta = r-r'$. Let $n'=n-\delta$ and $m'=m-\delta$. Clearly we have $n' < n$ and $m' < m$ with $m'-n' = m-n=k$. We also have $n \leq m'$ since this is equivalent to $\delta \leq k$, which follows from $\delta \leq r \leq k$. Finally, we have $n > \Delta_k+r \geq \Delta_k+\delta$, hence $n' > \Delta_k$. It follows
\[\alpha(n',m') = \Delta_k + (n'-\Delta_k-1 \bmod k+1) = \Delta_k + (r-\delta \bmod k+1) = \Delta_k + r' = \ell.\]
Now let us assume $\ell < \Delta_k$ and choose $0 \leq k' < k$ such that $\Delta_{k'} \leq \ell < \Delta_{k'+1}$. There is a unique integer $n'$ such that $n-k' \leq n' \leq n$ and $n' \equiv \ell+1 \bmod k'+1$. We have $n' \geq n-k' > 0$ because of $k' < k \leq \Delta_k < n$. Define $m' = n' + k'$. Then $n' \leq n$ and $n \leq n'+k'=m'$ hold by construction. We also have $m' < m$ because of $n' \leq n$ and $k' < k$. The inequality $n' > \Delta_{k'}$ follows from
\[n' \geq n-k' > \Delta_k-k' \geq \Delta_{k'+1}-k' = \Delta_{k'}+1.\]
We conclude
\[\alpha(n',m') = \Delta_{k'} + (n' - \Delta_{k'} - 1 \bmod k'+1) = \Delta_{k'} + (\ell - \Delta_{k'} \bmod k'+1) = \ell.\]

\textsc{2.\ Step.} We prove that $\Delta_k+r$ does not arise as $\alpha(n',m')$. In fact, if $n' \leq \Delta_{k'}$ (with $k'  \coloneqq  m'-n'$), then $\alpha(n',m') = n'+m' \geq m' \geq n > \Delta_k + r$. Else, if $n' > \Delta_{k'}$ and $\alpha(n',m') = \Delta_k+r$, then $\alpha(n',m') \in [\Delta_{k'},\Delta_{k'+1}[$ implies $k'=k$. It also implies
\[n'-\Delta_k - 1 \equiv r \equiv n-\Delta_k-1 \bmod k+1,\]
hence $n' \equiv n \bmod k+1$. Since $n' \leq n \leq m' = n' + k$, this implies $n=n'$ and then $m=m'$, a contradiction.
 \end{proof}

\begin{rem}
The next step would be to compute the nimber of $3$-generated abelian $p$-groups $\Z/p^{n_1} \oplus \Z/p^{n_2} \oplus \Z/p^{n_3}$ (with $n_1 \leq n_2 \leq n_3$). Let $k \coloneqq n_3-n_2$. Numerical experiments have suggested the following formula for the nimber:
\[ \left\{\begin{array}{cl} n_1 + n_2 + n_3 & \text{if }n_2 \leq \Delta_{k+n_1}, \\
n_1 + n_2 - 1 & \text{if }\Delta_{k+n_1} < n_2 \leq \Delta_{k+n_1+1}, \\
\Delta_{k+n_1} + ((n_2 - \Delta_{k+n_1} - 1) \bmod (k+n_1+1)) & \text{if }n_2 > \Delta_{k+n_1+1}. \end{array}\right.\]
For $n_2 > \Delta_{k+n_1+1}$ the formula seems to be fine, but for $n_2 \leq \Delta_{k+n_1+1}$ there are (for fixed $n_1$ only a few) exceptions.
\end{rem}

\begin{rem} \label{noform}
Assume that we had found a formula for the nimber of an arbitrary finite abelian $p$-group. According to Proposition \ref{redp} the game of an arbitrary finite abelian group is the selective compound of games of finite abelian $p$-groups. However, this does not directly allow us to compute the nimber of an arbitrary finite abelian group. This is because the nimber of a selective compound game does not have to only depend on the nimbers of the individual games. For example, let $G=H=\ast 1$ be two Nim-piles of size $1$. Then $\alpha(G)=\alpha(H)=1$. The options of $G \vee H$ are $\ast 0,G,H$, so that $\alpha(G \vee H)=2$. If we replace $H$ by the game $H'$ which has an additional option $\ast 2$, then we still have $\alpha(H')=1$, but one computes, in this order, $\alpha(G \vee \ast 1)=2$, $\alpha(G \vee \ast 2)=3$ and $\alpha(G \vee H')=4$. Notice that, however, $H'$ does not arise as the game of a finite abelian group. For the sake of completeness, let us mention that for natural numbers $n,m$ one has $\alpha(\ast n \vee \ast m) = n + m$, and that in case of infinite ordinals $n,m$ we have to replace $n+m$ by the Hessenberg sum $n \,\#\, m$ \cite{H}.
\end{rem}
 
However, there \emph{is} a method which reduces the game of an arbitrary finite abelian group to the game of a finite abelian $p$-group. By Proposition \ref{equiv} we only have to look at the game of divisor sequences.
  
\begin{prop} \label{reducep}
Let $p$ be any prime number. Then the game of any divisor sequence $n_1 \mid \cdots \mid n_s$, where $n_i \geq 1$, is equivalent to the game of the divisor sequence $p^{\Omega(n_1)} \mid \cdots \mid p^{\Omega(n_s)}$, i.e.\ the nimbers coincide.
\end{prop}

\begin{proof}
We prove this via induction. The nimber of $n_1 \mid \cdots \mid n_s$ is the mex of the nimbers of divisor sequences $m_1 \mid \cdots \mid m_s$ with $m \neq n$ and $m_1 \mid n_1 \mid m_2 \mid \cdots \mid n_s$. We define $n_0 \coloneqq 1$, so that this condition reads $n_{i-1} \mid m_i \mid n_i$ for $i=1,\dotsc,s$. By induction hypothesis, the nimber of $m_1 \mid \cdots \mid m_s$ equals the nimber of $p^{\Omega(m_1)} \mid \cdots \mid p^{\Omega(m_s)}$. Since $n_{i-1} \mid n_i$ for $i=1,\dotsc,s$, we have $\Omega(n_{i-1}) \leq \Omega(n_{i})$. If $d \geq 1$ is such that $n_i \mid d \mid n_{i+1}$, then $\Omega(n_{i-1}) \leq \Omega(d) \leq \Omega(n_i)$. Conversely, for every $\Omega(n_{i-1}) \leq k \leq \Omega(n_i)$ there is some $d \geq 1$ satisfying $n_i \mid d \mid n_{i+1}$ and $\Omega(d)=k$; a similar argument has been given in Lemma \ref{Zn}. This shows that the set of the divisor sequences $p^{\Omega(m_1)} \mid \cdots \mid p^{\Omega(m_s)}$ with $n_{i-1} \mid m_i \mid n_i$ coincides with the set of the divisor sequences $p^{k_1} \mid \cdots \mid p^{k_s}$ with $\Omega(n_{i-1}) \leq k_i \leq \Omega(n_i)$, i.e.\ $p^{\Omega(n_{i-1})} \mid p^{k_i} \mid p^{\Omega(n_i)}$. The mex of their nimbers is the nimber of the divisor sequence $p^{\Omega(n_1)} \mid \cdots \mid p^{\Omega(n_s)}$.
\end{proof}

\begin{cor} \label{nimbred}
Let $n_1 \mid \cdots \mid n_s$ be a divisor sequence with $n_i \geq 1$. Let $p$ be any prime number. Then the nimber of the abelian group $\Z/n_1 \oplus \cdots \oplus \Z/n_s$ equals the nimber of the abelian $p$-group $\Z/p^{\Omega(n_1)} \oplus \cdots \oplus \Z/p^{\Omega(n_s)}$.
\end{cor}
 
\begin{proof}
This follows from Propositions \ref{reducep} and \ref{equiv}.
\end{proof}

\begin{prop} \label{infinnim}
For all natural numbers $n \geq 1$ we have $\alpha(\Z/n \oplus \Z)=\omega+\Omega(n)$.
\end{prop}

\begin{proof}
We will prove this via induction on $n$. Let us assume that the claim is true for all positive natural numbers $<n$. The options of $\Z/n \oplus \Z$ are on the one hand $\Z/m \oplus \Z$ with $m \mid n$ and $m < n$, which have nimbers $\omega+\Omega(m)$ by induction hypothesis, and on the other hand the finite abelian groups $(\Z/n \oplus \Z)/\langle (z,u)\rangle$ with $z \in \Z/n$ and $0 \neq u \in \Z$, which have nimbers $<\omega$. We have to show that the latter nimbers actually cover all natural numbers. This will be already true for $z=0$ and $u=m p^k$ for $k \geq 0$, $m \mid n$ and some prime number $p$ which is coprime to $n$. In that case, the abelian group is isomorphic to $\Z/n \oplus \Z/u \cong \Z/m \oplus \Z/n p^k$. By Corollary \ref{nimbred} its nimber equals that of $\Z/p^{\Omega(m)} \oplus \Z/p^{\Omega(n)+k}$, which has been computed in Theorem \ref{nimbtwo}. Since $\Omega(m) \leq \Omega(n)$ and $k \geq 0$ can be chosen arbitrarily, it is readily checked that all natural numbers appear.
\end{proof}

\begin{rem}
We already know that $\Z \oplus \Z$ is $\P$ and therefore has nimber $0$. Proposition \ref{infinnim} in conjunction with Example \ref{Bezout} gives a more precise result, namely that the nimbers of the options of $\Z \oplus \Z$ are the ordinal numbers in the interval $\left[\omega,\omega+\omega\right[$. Theorem \ref{nimbtwo}, Corollary \ref{nimbred} and Proposition \ref{infinnim} give a complete calculation of the nimbers of $2$-generated abelian groups.
\end{rem}
 
\begin{rem}
There is a general upper bound of the nimbers: If $A \cong A_t \oplus \Z^r$ is any finitely generated abelian group of rank $r$, then $\alpha(A) \leq \omega \cdot r + \ell(A_t)$, where $\ell$ denotes the length of a finite $\Z$-module \cite[Chapter 6]{AM69}. The length is an additive function satisfying $\ell(\Z/n)=\Omega(n)$. The inequality can be proven by an induction which is similar to the case analysis in the proof of Theorem \ref{fg}. In particular, we have $\alpha(A)<\omega^2$.
\end{rem}
 
 
\section{The game of groups} \label{sec:grp}

\subsection{Some examples of groups}

In this section we will consider the game of (non-abelian) groups under the normal play rule. In every move, a group $G$ is replaced by the quotient group $G/\clos{a}$ for some $1 \neq a \in G$. If some option in this game happens to be abelian, then we continue with the game of abelian groups which has already been discussed in Section \ref{sec:ab}. However, in the non-abelian case, the normal subgroup $\clos{a}$ generated by $a$  tends to be quite large compared to the cyclic subgroup $\langle a \rangle$. This will be responsible for a variety of $\N$-positions in the game of groups. In fact, there are many non-trivial groups which can be normally generated by a single element \cite{B}, which are therefore $\N$.
  
\begin{ex} \label{knot}
Every knot group is normally generated by a single element. For example, the Wirtinger presentation of the trefoil knot is
\[G = \langle a,b,c : a^{-1} c a=b, c^{-1} b c = a, b^{-1} ab=c \rangle,\]
and we see $G= \clos{a}$.
\end{ex}
 
\begin{ex} \label{symgrp}
If $n \geq 2$, then the symmetric group $S_n$ is normally generated by $(1\, 2)$. For $n \geq 3$ the alternating group $A_n$ is normally generated by $(1\, 2\, 3)$. Hence, $S_n$ and $A_n$ are $\N$.
\end{ex}
 
\begin{ex} \label{dihedral}
If $n \geq 3$, then the dihedral group
\[D_n = \langle r,s: r^n = s^2 = (rs)^2 = 1\rangle\]
is $\N$: If $n$ is even, then $D_n / \clos{ r^2 } \cong (\Z/2)^2$ is a square of an abelian group and hence $\P$. If $n$ is odd, then $D_n / \clos{ s }$ is trivial and hence $\P$. We note that the product $D_n \times \Z/2$ is also $\N$ because the quotient by $(r,0)$ is isomorphic to $(\Z/2)^2$, which is $\P$.
\end{ex}

\begin{ex} \label{dicyclic}
The dicyclic group $\Dic_n$ of order $4n$ is defined by the presentation
\[\Dic_n = \langle a,x : a^{2n} = 1,\, a^n = x^2,\, axa=x \rangle.\]
For $n=2$ this is the Quaternion group $Q=\{\pm 1,\pm i,\pm j,\pm k\}$. If $n \geq 2$, then $\Dic_n$ is $\N$: If $n$ is odd, then $\Dic_n / \clos{ x }$ is trivial and hence $\P$. If $n$ is even, then $\Dic_n / \clos{ a } \cong (\Z/2)^2$, which is $\P$. We note that the product $\Dic_n \times \Z/2$ is also $\N$ because the quotient by $(a,0)$ is isomorphic to $(\Z/2)^2$.
\end{ex}
  
\begin{ex} \label{pq}
Let $p,q$ be two distinct primes and let $G$ be a group of order $pq$. Then $G$ is $\N$: If $G$ is abelian, then $G$ is isomorphic to $\Z/p \times \Z/q$, which is $\N$. If $G$ is not abelian, then it is well-known that $G = \langle x,y : x^q = y^p = 1 ,\, y x y^{-1} = x^r\rangle$ holds for some $\overline{r} \in (\Z/q)^{\times}$ of order $p$. In particular, $q$ and $r-1$ are coprime. But then $G/\clos{ y } = \langle x : x^q = x^{r-1} = 1 \rangle$ is trivial.
\end{ex}

\subsection{Groups of small order}

All non-abelian groups we have encountered so far are $\N$. We will now use the classification of groups of small order to find the smallest examples of non-abelian groups which are $\P$. There are various online resources for this classification such as {\footnotesize\texttt{http://groupprops.subwiki.org/wiki/Category:Groups_of_a_particular_order}}. For the general theory and development of this classification, we refer to \cite{HBE}.
  
\begin{prop}
Every non-abelian group of order $\leq 15$ is $\N$.
\end{prop}
 
\begin{proof}
We have already dealt with groups of order $pq$ for primes $p,q$ in Example \ref{pq}, and groups of prime order are cyclic. This only leaves the orders $8$ and $12$. There are $2$ non-abelian groups of order $8$, namely the dihedral group $D_4$ and the quaternion group $Q$, which are $\N$ by Examples \ref{dihedral} and \ref{dicyclic}. There are $3$ non-abelian groups of order $12$, namely $A_4$, $D_6$ and $\Dic_3$, which are also $\N$ by Examples \ref{symgrp}, \ref{dihedral} and \ref{dicyclic}.
\end{proof}
 
Next, there are $14$  groups of order $16$ \cite{W} (up to isomorphism, of course). We denote them via their IDs in GAP's SmallGroup library (\texttt{http://www.gap-system.org}). Thus, $G_n$ is encoded by \texttt{SmallGroup(16,n)}. Since $G_1,G_2,G_5,G_{10},G_{14}$ are abelian, we only need to consider the other $9$ non-abelian groups. In the following list, $G \rtimes_{\varphi} N$ denotes the semidirect product associated to a homomorphism $\varphi : G \to \Aut(N)$.
\begin{itemize}[noitemsep,leftmargin=4ex]
\item $G_3 = \langle a,b,c : a^4=b^2=c^2=1, ab=ba, bc=cb, cac^{-1}=ab \rangle$\\
$\hphantom{G_3} = (\Z/4 \times \Z/2) \rtimes_{\varphi} \Z/2$ with $\varphi(c) = (a \mapsto ab, b \mapsto b)$.
\item $G_4 = \langle a,b : a^4 = b^4 = 1, ab = ba^3 \rangle = \Z/4 \rtimes_3 \Z/4$
\item $G_6 = \langle a,b : a^8 = b^2 = 1, ab = ba^5 \rangle = \Z/8 \rtimes_5 \Z/2$
\item $G_7 = D_8$
\item $G_8 = \langle a,b : a^8 = b^2 = 1, ab = ba^3 \rangle =  \Z/8 \rtimes_3 \Z/2$
\item $G_9 = \Dic_4$
\item $G_{11} = D_4 \times \Z/2$
\item $G_{12}= \Dic_2 \times \Z/2$
\item $G_{13} = \langle a,x,y : a^4=x^2=1, a^2 = y^2, xax=a^{-1} , ay=ya, xy=yx \rangle$
\end{itemize}

We already know that $G_7,G_9,G_{11},G_{12}$ are $\N$ by Examples \ref{dihedral} and \ref{dicyclic}. Observe that $G_6 / \clos{ a^2 } = \langle a,b : a^2 = b^2 = 1, ab=ba \rangle \cong (\Z/2)^2$. The same argument shows $G_8 / \clos{ a^2 } \cong (\Z/2)^2$. We also see $G_{13}/ \clos{ a } = \langle x,y : x^2 = y^2 = 1, xy=yx \rangle \cong (\Z/2)^2$. Thus, $G_6,G_8,G_{13}$ are $\N$. However, $G_3$, $G_4$ turn out to be $\P$. This can be verified by computing all quotients by hand. Alternatively, we may use the following simple GAP-program. It has a small group $G$ as an input and returns the set of structure descriptions of all quotients $G/ \clos{ g }$ for $1 \neq g \in G$.
\begin{verbatim}
Quotients := function(G)
local s,g,Q;  s := [];
for g in Elements(G){[2..Order(G)]} do
  Q := FactorGroup(G,NormalClosure(G,Subgroup(G,[g])));
  AddSet(s,StructureDescription(Q)); od;
return s;
end;
\end{verbatim}
With this program we may compute the quotients of $G_3$ and $G_4$: 
\begin{verbatim}
gap> Quotients(SmallGroup(16,3));
[ "C2", "C4", "C4 x C2", "D8" ]
gap> Quotients(SmallGroup(16,4));
[ "C2", "C4", "C4 x C2", "D8", "Q8" ]
\end{verbatim}
In our notation, these quotients are $\Z/2$, $\Z/4$, $\Z/2 \times \Z/4$, $D_4$ and $Q$, which have already been verified to be $\N$. Thus, $G_3$ and $G_4$ are $\P$. We have proven the following:

\begin{prop}
Among the $9$ non-abelian groups of order $16$, there are exactly $2$ which are $\P$, namely $G_3=(\Z/4 \times \Z/2) \rtimes_{\varphi} \Z/2$ and $G_4 = \Z/4 \rtimes_3 \Z/4$.
\end{prop}

\begin{rem}
In the same way we may proceed with other small group orders. Using GAP, we have verified that among the $6065$ groups of order $\leq 200$ only $105$ groups are $\P$, of which $86$ groups are non-abelian, namely:
\begin{itemize}[noitemsep,leftmargin=4ex]
\item $2$ groups of order $16$ with IDs $3,4$ already mentioned,
\item $1$ group of order $36$ with ID $13$,
\item $68$ groups of order $64$ with IDs $3,\dotsc,16$,$56$,$193,\dotsc,245$,
\item $2$ groups of order $81$ with IDs $3,4$,
\item $1$ group of order $100$ with ID $15$,
\item $2$ groups of order $128$ with IDs $175$, $476$,
\item $9$ groups of order $144$ with IDs $92,93,94,95,100,102,103,194,196$,
\item $1$ group of order $196$ with ID $11$.
\end{itemize}
\end{rem}

\subsection{The game of subgroups} \label{subsec:subgroups}

The game of groups from the previous subsection has disproportionally many $\N$-positions because the normal closure of an element is rather large. Therefore, we propose and briefly sketch a different, more balanced game:
 
We start with a group $G$. A position in the game of subgroups is a subgroup $U \subseteq G$. The initial position is the trivial subgroup, and the terminal position is the whole group. A move picks some $g \in G \setminus U$ and replaces $U$ by the subgroup $\langle U,g \rangle$. Thus, a sequence of moves is given by elements $g_1,g_2,\dotsc$ of $G$ such that $g_{i+1}$ is not contained in the subgroup $\langle g_1,\dotsc,g_i\rangle$ generated by the previous elements. The ending condition is satisfied if and only if $G$ is Noetherian, i.e.\ every subgroup of $G$ is finitely generated. For example, this happens when $G$ is finite. Let us restrict to the normal play rule. When is $G$ a $\P$-position? By this we actually mean that the trivial subgroup is a $\P$-position in the game of subgroups of $G$.

Remark that this resembles the game proposed in \cite{AH}, whose positions are the subsets of $G$. Our game is also related to the game of algebraic structures in the special case of $G$-sets, starting with the $G$-set $G$. In fact, for a subgroup $U \subseteq G$, a move from the $G$-set $G/U$ picks some $g \in G \setminus U$ and replaces $G/U$ by the $G$-set $G/\langle U,g\rangle$. The only difference between the two games is the following: Two $G$-sets $G/U$, $G/V$ are isomorphic if and only if $U,V$ are conjugate, not necessarily equal.
   
Observe that when $G$ is abelian, we get the game of the abelian group $G$ and we may use Theorem \ref{fin} to predict the game outcome. When $G$ is Hamiltonian (i.e.\ every subgroup is normal), we have the game of the group $G$. But for arbitrary $G$, these games differ dramatically, because many more $\P$-positions arise.
 
For example, $D_3 = S_3$, $D_5$ and $A_4$ are $\P$. However, $D_4$ and $D_6$ are $\N$. Let us verify this for $S_3$: If Player I starts with some $2$-cycle (resp.\ $3$-cycle), then Player II responds with any $3$-cycle (resp.\ $2$-cycle). Since a $2$-cycle and a $3$-cycle already generate $S_3$, Player II wins. The quaternion group $Q$ is $\N$ as before because it is Hamiltonian.
 
The subgroup structure of dihedral groups is quite easy and may be used to find the game outcome:
  
\begin{prop}\label{dihe}
Let $n \geq 1$. In the game of subgroups, the dihedral group $D_n$ is $\P$ if and only if $n$ is a prime number.
\end{prop}

\begin{proof} Clearly, $D_1 \cong \Z/2$ is $\N$ and $D_2 \cong (\Z/2)^2$ is $\P$. Now let us assume $n \geq 3$. If $r$ denotes the rotation and $s$ denotes the reflection, the subgroups of $D_n$ are the following:
\begin{itemize}[noitemsep,leftmargin=4ex]
\item $U_d  \coloneqq  \langle r^d \rangle$ for $d \mid n$
\item $U_{d,i}  \coloneqq  \langle r^d, r^i s \rangle$ for $d \mid n$ and $0 \leq i < n$
\end{itemize}
Now let us suppose first that $n$ is a prime number. Then Player I can only make the moves $U_1 = \langle r \rangle$ or $U_{n,i} = \langle r^i s \rangle$. In the first case, Player II answers with $s$; in the second case he answers with $r$. In each case, Player II arrives at $\langle r,s \rangle = D_n$ and wins.

Now let $n$ be not a prime number. Choose some prime factor $p\mid n$. The winning move for Player I is $U_p = \langle r^p \rangle$: This is a normal subgroup, so that Player II continues with the game of subgroups of $D_n / U_p \cong D_p$, which we already know is $\P$.
\end{proof}

\begin{prop}
Let $n \geq 1$. In the game of subgroups, the symmetric group $S_n$ is $\P$ if and only if $n \neq 2$.
\end{prop}

\begin{proof}
We already know this for $n \leq 3$. For $n>4$ it is known that $S_n$ is $\frac{3}{2}$-generated \cite{B70,IZ95}, i.e.\ that for every $1 \neq g \in S_n$ there is some $h \in S_n$ with $S_n = \langle g,h \rangle$. It is then automatic that $h \notin \langle g \rangle$. In other words, for every move by Player I there is a winning move for Player II, showing that $S_n$ is $\P$. Although $S_4$ is not $\frac{3}{2}$-generated, it is $\P$ as well: Every element of $S_4$ is conjugated to one of the elements $(1~2)$, $(1 ~ 2 ~ 3)$, $(1 ~ 2 ~ 3 ~ 4)$, $(1~2) (3 ~ 4)$, so that we may assume that Player I chooses one of them. In the first three cases, Player II can immediately win by producing a generating set, namely by choosing $(1~2~3~4)$ in the first case, $(1~4)$ in the second case, and $(1~2)$ in the third case. In the last case, Player II responds with $(1~3)(2 ~ 4)$. This produces a normal subgroup isomorphic to $\Z/2 \oplus \Z/2$, whose quotient group is isomorphic to $S_3$. Since $S_3$ is $\P$, it follows that Player II makes the last move.
\end{proof}


\section{The game of commutative rings} \label{sec:ring}

\subsection{Some examples of commutative rings}

In this section we will study the game of commutative rings; therefore we will require some basics of commutative ring theory \cite{AM69}. The game starts with some commutative ring $R$, and a move consists of choosing some element $a \in R \setminus \{0\}$ and replacing $R$ by $R/\langle a \rangle$, where $\langle a \rangle$ denotes the principal ideal generated by $a$. We have already observed that the ending condition holds precisely for Noetherian commutative rings and that all non-trivial rings are normal $\N$-positions, which is why we will concentrate on the mis\`{e}re play  rule. The examples in this section are mainly motivated by the modest goal to decide whether polynomial rings are $\N$ or $\P$.
 
\begin{rem}
We have already seen in Corollary \ref{ringprod} that if $R$ is a $\P$-position in the game of commutative rings, then $R$ cannot be written as a product of two non-trivial rings.
\end{rem}
 
\begin{rem}
The duality between commutative rings and affine schemes \cite{GW} shows that the game of commutative rings is equivalent to a game of affine schemes: The options of a Noetherian affine scheme are the closed subschemes which are cut out by some single non-zero global section. The game ends with the empty scheme. This viewpoint is quite useful to get some geometric intuition for the game, and we will use it a couple of times. Corollary \ref{ringprod} says that every $\P$-position is a connected affine scheme. Since the dimension of a closed subscheme is less or equal, typically less than the dimension of the whole scheme, in order to solve the game for higher-dimensional schemes one first has to look at schemes of low dimensions such as $0$ and $1$. This is what we will do next.
\end{rem}

\begin{ex} \label{fld}
The zero ring $0$ is $\N$. Fields are $\P$, because $0$ is the only option.
\end{ex}
 
\begin{ex}\label{PIR}
Let $R$ be a Noetherian commutative ring. If $R$ has a principal maximal ideal $\neq 0$, then $R$ is $\N$. The winning move is to quotient out the maximal ideal, which yields a field.

This applies in particular to principal ideal rings (not necessarily domains) which are no fields, such as $\Z$, the polynomial ring $K[X]$ over a field $K$, and quotients thereof such as $\Z/4$.

It also shows, for example, that $K[X,Y]/\langle XY\rangle$ and $K[X,Y]/\langle XY-1 \rangle$ are $\N$. The winning move is to quotient out $X-1$ in each case. In the corresponding game of affine schemes, this means that we intersect the union of the coordinate axes resp.\ the standard hyperbola with the line $X=1$, which results in a single simple point in each case, which is therefore $\P$. The following picture illustrates this.

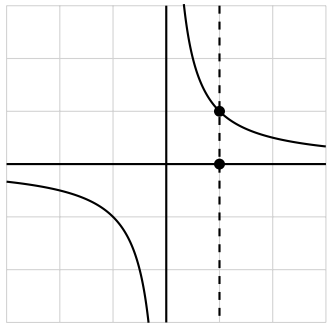
\begin{center}
\begin{tikzpicture}[scale=0.7,thick]
\draw [lightgray!70,thin,step=1] (-3,-3) grid (3,3);
\draw plot[variable=\t,samples=\snr,domain=0.33:3] (\t,{1/\t)});
\draw plot[variable=\t,samples=\snr,domain=-3:-0.33] (\t,{1/\t)});
\draw (-3,0) to (3,0);
\draw (0,-3) to (0,3);
\draw [dashed] (1,3) to (1,-3);
\fill (1,1) circle (0.25em);
\fill (1,0) circle (0.25em);
\end{tikzpicture}
\end{center}
\end{ex}

\begin{ex}
If $p$ is a prime, then up to isomorphism there are four rings with $p^2$ elements (remember that rings are unital by definition), which are automatically commutative, namely $\IF_{p^2}$, $\Z/p^2$, $\IF_p \times \IF_p$ and $\IF_p[X]/\langle X^2 \rangle$ \cite{F93}. By the previous results, they are all $\N$ except of course for the field $\IF_{p^2}$.
\end{ex}

Let us continue with $1$-dimensional examples. Recall from \cite[Chapter 9]{AM69} that a Dedekind domain is an integrally closed Noetherian integral domain of Krull dimension $1$. 
 
\begin{prop}\label{dedekind}
Let $R$ be a Dedekind domain. If $R$ has some principal maximal ideal, then $R$ is $\N$. Otherwise, $R$ is $\P$.
\end{prop}

\begin{proof}
The first part has  already been observed in Example \ref{PIR}. Now let us assume  that $R$ has no principal maximal ideal. If $0 \neq a \in R$, then $R/\langle a \rangle$ is $\N$: We may assume that $a$ is not a unit. Then there is some maximal ideal $I \subseteq R$ containing $a$. By \cite[Section I.1, Corollary 2 to Theorem 4]{FT} there is some $b \in I$ such that $I = \langle a,b \rangle$. Since $I$ is not principal, we have $b \notin \langle a \rangle$. Hence, $R/\langle a,b \rangle=R/I$ is a field which is an option of $R/\langle a \rangle$.
\end{proof}

From this result and the basic theory of elliptic curves \cite{K92} we can derive the first $2$-dimensional example:
  
\begin{prop} \label{weier}
Let $K$ be an algebraically closed field. If $f=0$ is any affine Weierstrass equation in $K[X,Y]$, then $K[X,Y]/\langle f \rangle$ is $\P$. Hence, $K[X,Y]$ is $\N$.
\end{prop}
 
\begin{proof}
Let $E$ be the elliptic curve over $K$ corresponding to $f$ and let $\infty \in E$ be the point at infinity. Let $R=K[X,Y]/\langle f \rangle$, so that $E \setminus \{\infty\}=\Spec(R)$. Since $E$ is a smooth curve, $R$ is a Noetherian, $1$-dimensional integral domain whose localizations at maximal ideals are discrete valuation domains. Hence, $R$ is a Dedekind domain. Moreover, it has no principal maximal ideal; this is a consequence of the Riemann-Roch Theorem. Hence, $R$ is $\P$ by Proposition \ref{dedekind}.
\end{proof}

\begin{ex}
Explicit examples of affine Weierstrass equations include $Y^2=X^3+1$  if $\chara(K) \neq 3$ and $Y^2=X^3-X$ if $\chara(K) \neq 2$. Here is an example of a game starting with $K[X,Y]$. Player I wins.
\begin{align*}
K[X,Y]  \I \, & K[X,Y]/\langle Y^2-X^3-1 \rangle \II K[X,Y]/\langle Y^2-X^3-1,3XY-1\rangle  \\
& \cong K[X]/\langle X^5+X^2 - \tfrac{1}{9} \rangle \cong K^5 \I K^5/\langle (1,1,1,1,0) \rangle \cong K \II 0.
\end{align*}
Geometrically, this game looks as follows. Obviously we did not draw the two non-real points of intersection.

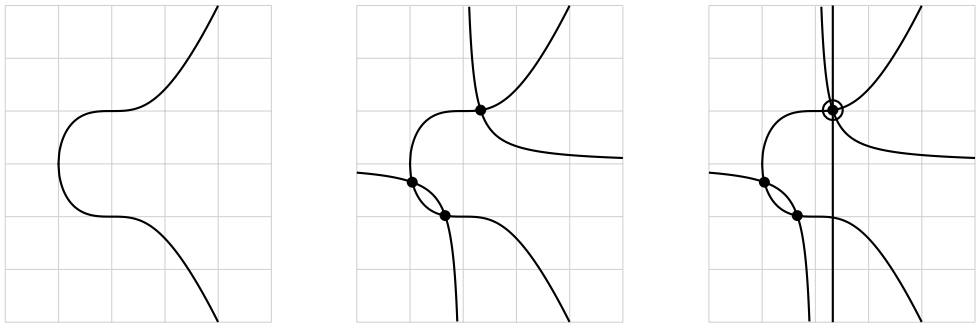
\begin{center}
\begin{tikzpicture}[scale=0.7,thick]
\draw [lightgray!70,thin,step=1] (-2,-3) grid (3,3);
\draw plot[variable=\t,samples=\snr,domain=-1:2,smooth] (\t,{+sqrt(\t^3+1)});
\draw plot[variable=\t,samples=\snr,domain=-1:2,smooth] (\t,{-sqrt(\t^3+1)});
\draw (-0.999,-0.01) to (-0.999,+0.01);
\end{tikzpicture}
\hspace{2em}
\begin{tikzpicture}[scale=0.7,thick]
\draw [lightgray!70,thin,step=1] (-2,-3) grid (3,3);
\draw plot[variable=\t,samples=\snr,domain=-1:2,smooth] (\t,{+sqrt(\t^3+1)});
\draw plot[variable=\t,samples=\snr,domain=-1:2,smooth] (\t,{-sqrt(\t^3+1)});
\draw (-0.999,-0.01) to (-0.999,+0.01);
\draw plot[variable=\t,samples=\snr,domain=0.112:3,smooth] (\t,{1/(3*\t)});
\draw plot[variable=\t,samples=\snr,domain=-2:-0.11,smooth] (\t,{1/(3*\t)});
\fill (-0.34,-0.9801) circle (0.25em);
\fill (0.3276,1.0174) circle (0.25em);
\fill (-0.9578,-0.3480) circle (0.25em);
\end{tikzpicture}
\hspace{2em}
\begin{tikzpicture}[scale=0.7,thick]
\draw [lightgray!70,thin,step=1] (-2,-3) grid (3,3);
\draw plot[variable=\t,samples=\snr,domain=-1:2,smooth] (\t,{+sqrt(\t^3+1)});
\draw plot[variable=\t,samples=\snr,domain=-1:2,smooth] (\t,{-sqrt(\t^3+1)});
\draw (-0.999,-0.01) to (-0.999,+0.01);
\draw plot[variable=\t,samples=\snr,domain=0.112:3,smooth] (\t,{1/(3*\t)});
\draw plot[variable=\t,samples=\snr,domain=-2:-0.11,smooth] (\t,{1/(3*\t)});
\fill (-0.34,-0.9801) circle (0.25em);
\fill (0.3276,1.0174) circle (0.25em);
\fill (-0.9578,-0.3480) circle (0.25em);
\draw  (0.3276,-3) to (0.3276,3);
\draw (0.3276,1.0174) circle (0.45em);
\end{tikzpicture}
\end{center}
\end{ex}
 
Actually, $K[X,Y]$ is $\N$ for \emph{every} field $K$. We will give a much more elementary proof which does not use any algebraic geometry later in Corollary \ref{pol}.
 
\subsection{Zero-dimensional rings}

After having considered smooth curves, the next step is to consider an example of a non-smooth curve such as the cuspidal cubic curve whose coordinate ring is $K[X,Y]/\langle Y^2-X^3 \rangle$. We will show that it is $\P$, giving another reason why $K[X,Y]$ is $\N$. However, we will need some results on zero-dimensional rings first, which appear as intersections of the cuspidal curve with other curves through the origin.
  
\begin{lem} \label{kV}
Let $V$ be a finite-dimensional vector space over some field $K$. Then the ring $K \oplus V$ with unit $1 \oplus 0$ and multiplication $V \cdot V=0$ is $\N$ if and only if $\dim(V)$ is odd. Otherwise it is $\P$.
\end{lem}
 
\begin{proof}
For $V=0$ this is true. We make an induction on $\dim(V)$. If $\dim(V)$ is odd, choose some $0 \neq v \in V$. The ideal generated by $0 \oplus v$ equals $0 \oplus Kv$, and the quotient is $K \oplus V/Kv$, which is $\P$ by the induction hypothesis. Now let us assume that $\dim(V)$ is even and $0 \neq a \oplus v \in K \oplus V$ is some element. If $a \neq 0$, then $a \oplus v$ is invertible with $(a \oplus v)^{-1} = a^{-1} \oplus {-}a^{-2} v$, so that the quotient is zero, which is $\N$. Otherwise, $a=0$ and the quotient is $K \oplus V/Kv$, which is $\N$ by induction hypothesis.
\end{proof}

\begin{cor} \label{qu}
If $K$ is a field, then $K[X,Y]/\langle X^2,XY,Y^2\rangle$ is $\P$.
\end{cor}

\begin{proof}
This is the  special case of Lemma \ref{kV} with $\dim(V)=2$.
\end{proof}

\begin{lem} \label{u}
Let $K$ be a field and $n \geq 0$ be any natural number. Then the ring $K[X,Y]/\langle Y^2-X^3,X^{n+1},X^n Y\rangle$ is $\P$.
\end{lem}

\begin{proof}
Let us call this ring $B_n$. Then $B_0=K$ and $B_1= K[X,Y]/\langle Y^2,X^2,XY\rangle$ are $\P$ by Example \ref{fld} and Corollary \ref{qu}. Now let $n \geq 2$ and let us assume that the claim holds for all natural numbers $<n$. Observe that $1,\dotsc,X^n,Y,XY,\dotsc, X^{n-1} Y$ is a $K$-basis of $B_n$. Choose some non-zero element $b \in B$, we want to show that $Q \coloneqq B_n / \langle b \rangle$ is $\N$. Let us write
\[b = r_0 + r_1 X + \cdots + r_n X^n + s_1 Y + \cdots + s_n X^{n-1} Y\]
with elements $r_i,s_j \in K$ which are not all zero. If $r_0 \neq 0$, then $b$ is a unit and we are done. Let $r_0=0$. Choose some minimal $1 \leq d \leq n$ with $r_i=s_i=0$ for all $1 \leq i < d$. Thus, we have
\[b = r_d X^d + \cdots + r_n X^n + s_d X^{d-1} Y + \cdots + s_n X^{n-1} Y,\]
and one of $r_d,s_d$ is non-zero. Consider the case $d=n$, so that $b = r_n X^n + s_n X^{n-1} Y$. If $r_n = 0$, we have $Q = K[X,Y]/\langle Y^2-X^3,X^{n+1},X^{n-1} Y\rangle$ and therefore $Q/\langle X^n \rangle \cong B_{n-1}$, which is $\P$ by the induction hypothesis. This proves that $Q$ is $\N$. If $r_n \neq 0$, then $X^{n-1} Y \neq 0$ holds in $Q$ and we have $Q/\langle X^{n-1} Y \rangle \cong B_{n-1}$, which is $\P$ by the induction hypothesis.

So let us assume $d <n$. In $B_n$ we compute:
\begin{eqnarray*}
X^{n-d-1} b & = & r_d X^{n-1} + r_{d+1} X^n + s_d X^{n-2} Y + s_{d+1} X^{n-1} Y \\
X^{n-d} b & = & r_d X^n + s_d X^{n-1} Y \\
X^{n-d-1} Y b & = & r_d X^{n-1} Y
\end{eqnarray*}
Now we compute in the quotient $Q$, where $b=0$. When $r_d \neq 0$ in $K$, the third equation shows $X^{n-1} Y = 0$ in $Q$, which in turn gives $X^n = 0$ by the second equation. But then $b$ lifts to an element $b' \in B_{n-1}$ and we obtain $Q \cong B_{n-1} / \langle b' \rangle$, which is $\N$ by the induction hypothesis. When $r_d = 0$, we have $s_d \neq 0$, so that the second equation gives $X^{n-1} Y = 0$, and the first equation reads as $r_{d+1} X^n + s_d X^{n-2} Y = 0$. We see $X^{n-1} \neq 0$ in $Q$ and $Q/\langle X^{n-1} \rangle \cong B_{n-2}$, which is $\P$ by the induction hypothesis, so that $Q$ is $\N$.
\end{proof} 

\begin{cor} \label{uc}
Let $K$ be a field and $n \geq 0$. Then $K[X,Y]/\langle Y^2-X^3,X^n Y\rangle$ and $K[X,Y]/\langle Y^2-X^3,X^{n+1}\rangle$ are $\N$. For example, $K[X,Y]/\langle X^3,Y^2\rangle$ is $\N$.
\end{cor}

\begin{prop} \label{cus}
Let $K$ be an algebraically closed field. Then $K[X,Y]/\langle Y^2-X^3\rangle$ is $\P$.
\end{prop}

\begin{proof} Let $R \coloneqq K[X,Y]/\langle Y^2-X^3\rangle$ and consider some $0 \neq f \in R$, represented by some polynomial $f \in K[X,Y] \setminus \langle Y^2-X^3\rangle$ of $Y$-degree $\leq 2$. Our goal is to show that $R/\langle f \rangle$ is $\N$. We assume first that
$f \notin \langle X,Y\rangle$ and write
\[f = a_0 + a_1 X + a_2 X^2 + \cdots + b_0 Y + b_1 X Y + b_2 X^2 Y + \cdots\]
with $a_0 \neq 0$. We claim that $X$ is invertible in $R/ \langle f \rangle$. This is clear if $b_0=0$. Otherwise, let $g$ be the same polynomial as $f$, but with $a_0$ replaced by $-a_0$. In $R/ \langle f \rangle$ we have $0 = fg$ and in that product the $Y$ has disappeared, but the constant term is still $\neq 0$. Thus, we may repeat the argument. 

Since $X$ is invertible in $R/\langle f \rangle$, there is an isomorphism $R/ \langle f \rangle \cong (R_X)/ \langle f \rangle$, where $R_X$ denotes the localization at the element $X$. The normalization map $\pi : R \to K[T]$ defined by $X \mapsto T^2$ and $Y \mapsto T^3$ becomes an isomorphism when localized at $X$, so that $R/ \langle f \rangle \cong K[T]_T / \langle\pi(f)\rangle = K[T]/\langle \pi(f) \rangle$ and $\pi(f)$ is some polynomial of degree $\geq 2$. Now apply Example \ref{PIR} to conclude that $R/ \langle f \rangle$ is $\N$.
 
Now let us assume $f \in \langle X,Y\rangle$. The intersection $V(f) \cap V(Y^2-X^3) \subseteq \IA^2_K$ is zero-dimensional. Thus, $R/\langle f\rangle$ is a direct product of local Artinian rings. In order to show that it is $\N$, we may even assume that it is local by Corollary \ref{ringprod}. This means that there is a unique $\alpha \in K$ such that $\pi(f)(\alpha)=f(\alpha^2,\alpha^3)=0$. Since $f(0,0)=0$ it follows $\pi(f)=a T^d$ for some $d \geq 2$ and some $a \in K^{\times}$, which means $f = a X^n Y$ or $f = a X^{n+1}$ for some $n \geq 0$. Now the claim follows from Corollary \ref{uc}.
 \end{proof}
 
\begin{rem}
With the same method of the proof of Lemma \ref{u} one can prove that for every $n \geq 1$ the ring $K[X,Y]/\langle X^n,XY,Y^n\rangle$ is $\P$. In particular, $K[X,Y]/\langle X^n,Y^n\rangle$ and $K[X,Y]/\langle X^n,XY,Y^m\rangle$ are $\N$ for $n,m \geq 2$ and $n \neq m$. This is yet another instance of the theme that ``squares'' are $\P$.
\end{rem}

\subsection{Polynomial rings}

In this subsection we will find the game outcome of $K[X,Y]$, where $K$ is any field. It is useful to generalize this to $R[X]$, where $R$ is any principal ideal domain which is not a field. As in the previous subsection, we will have to study some zero-dimensional rings first.

\begin{prop} \label{1}
Let $R$ be a principal ideal domain and $p \in R$ be a prime element. Then, for every $n \geq 1$, the ring $R/p^n [X]/\langle X^2,p^{n-1} X\rangle$ is $\P$. In particular, $R/p^n [X]/\langle X^2\rangle$ and $R[X]/\langle X^2,p^{n-1} X\rangle$ are $\N$.
\end{prop}

\begin{proof}
Since $R/p^n \cong R_{\langle p\rangle}/p^n$, where $R_{\langle p\rangle}$ denotes the localization at the prime ideal $\langle p \rangle$, we may assume that $p$ is the only prime element of $R$ up to units. We make an induction on $n$. For $n=1$ the ring $R/p^n [X]/\langle X^2,p^{n-1} X\rangle$ is the field $R/p$, which is $\P$. Now let us assume that $n \geq 2$ and that the claim has been proven for all positive natural numbers $<n$. Let $u \in R/p^n [X]/\langle X^2,p^{n-1} X\rangle$ be an arbitrary non-zero element, say $u=a+bX$ with $a,b \in R$. We have to show that $Q \coloneqq R/p^n [X]/\langle X^2,p^{n-1} X,u\rangle$ is $\N$. This is trivial when $u$ is a unit, so let us assume the opposite, i.e.\ that $a$ is not a unit.
  
If $b=0$, then $u$ is associated to $p^d$ for some $1 \leq d < n$, and $Q=R/p^d [X]/\langle X^2\rangle$ is $\N$ because $p^{d-1} X \neq 0$ in $Q$ and $Q/\langle p^{d-1} X\rangle$ is $\P$ by the induction hypothesis. So let us assume $b \neq 0$. If $a=0$, then $u$ is associated to $p^k X$ for some unique $0 \leq k < n-1$, and $Q=R/p^n [X]/\langle X^2,p^k X\rangle$ is $\N$, since $0 \neq p^{k+1}$ in $Q$ and by the induction hypothesis $Q/p^{k+1} \cong R/p^{k+1}[X]/\langle X^2,p^k X\rangle$ is $\P$. So let us assume $a \neq 0$. Let $d \coloneqq v_p(a)$ and  $k \coloneqq v_p(b)$, where $v_p$ denotes the multiplicity of $p$. Then we may assume $1 \leq d < n$ and $0 \leq k < n-1$.
 
Assume that $a$ divides $b$, i.e.\ $d \leq k$. In $Q$ we compute $0=(a+bX)X=aX$, hence $bX=0$. Therefore, the relation $u=0$ simplifies to $a=0$. It follows that $Q = R/p^d[X]/\langle X^2\rangle$, which is again $\N$ by induction hypothesis. Now let us assume $d>k$. In $Q$ we have $p^{n-k-1} b X = 0$ and therefore $0=p^{n-k-1} u=p^{n-k-1} a$. Hence, we have $0=p^{n-k+d-1}$, but no smaller power of $p$ vanishes in $Q$. In particular $p^{k+1} \neq 0$, because $2(k+1)< n+d$ implies $k+1 < n-k+d-1$. Therefore we are allowed to move to $Q/\langle p^{k+1}\rangle \cong R/p^{k+1}[X]/\langle X^2,p^k X\rangle$, which is $\P$ by the induction hypothesis. Hence, $Q$ is $\N$.
\end{proof}

\begin{prop} \label{2}
Let $R$ be a principal ideal domain, which is not a field. Then $R[X]/\langle X^2\rangle$ is $\P$. Hence, $R[X]$ is $\N$.
\end{prop}

\begin{proof}
Let $u \in R[X]/\langle X^2\rangle$ be some non-zero element, say $u=a+bX$ for $a,b \in R$. We have to show that $Q \coloneqq R[X]/\langle X^2,u\rangle$ is $\N$. This is trivial when $u$ is a unit, so let us assume the opposite, i.e.\ that $a$ is not a unit. If $b=0$, then $Q=R/a[X]/\langle X^2\rangle$. If $a$ is a prime power up to a unit, $Q$ is $\N$ because of Proposition \ref{1}. If not, the Chinese Remainder Theorem implies that $Q$ is a non-trivial product of non-trivial rings and therefore also $\N$ by Corollary \ref{ringprod2}. So let us assume $b \neq 0$. If $a=0$, then we choose some prime $p$ and we write $b=p^n c$ for some $p \nmid c$ and $n \geq 0$. Then $c$ is invertible modulo $p^{n+1}$. It follows $Q/\langle p^{n+1}\rangle = R/p^{n+1}[X]/\langle X^2,p^n X\rangle$, which is $\P$ according to Proposition \ref{1}. Hence, $Q$ is $\N$.
 
Now let us assume $a,b \neq 0$. If there is some prime $p$ with $v_p(a) > v_p(b) \eqqcolon n$, then a similar argument as above shows that $Q/\langle p^{n+1}\rangle=R/p^{n+1}[X]/\langle X^2,p^n X\rangle$ is $\P$, so that $Q$ is $\N$. Now let us assume $v_p(a) \leq v_p(b)$ for all primes $p$, i.e.\ that $a$ divides $b$. In $Q$ we compute $0=(a+bX)X=aX$, hence $0=bX$, and the relation $u=0$ simplifies to $a=0$. Hence, $Q \cong R/a[X]/\langle X^2\rangle$ is $\N$ by what we have already seen before.
\end{proof}
 
\begin{ex}
Here is an example for the game of commutative rings starting with $\Z[X]$. The first player wins. He chooses the moves resulting from the proofs above.
\begin{align*}
\Z[X]  \I \,& \Z[X]/\langle X^2\rangle \II \Z[X]/\langle X^2,36\rangle  \I \Z[X]/\langle X^2,36,18X-8\rangle \\
 & \cong \Z/4[X]/\langle X^2,2X\rangle  \II \Z/4[X]/\langle X^2,2X,X+2\rangle \cong \Z/4 \I \Z/2 \II 0
\end{align*}
\end{ex}

\begin{cor} \label{pol}
Let $K$ be a field. Then the polynomial ring $K[X,Y]$ is $\N$.
\end{cor}

\begin{proof}
This follows from Proposition \ref{2} applied to $R=K[Y]$.
\end{proof}

We conjecture that also $K[X,Y,Z]$, in fact all polynomial rings $K[X_1,\dotsc,X_n]$ for $n \geq 1$ are $\N$, because it seems very unlikely that $K[X,Y,Z]/\langle f \rangle$ is $\N$ for \emph{all} non-zero polynomials $f$. But the computational effort to check this even for a single candidate $f$ seems to be huge, because there will be far more ``layers'' of backward induction than in the proofs for $K[X,Y]/\langle Y^2-X^3 \rangle$ and $K[X,Y]/\langle X^2\rangle$. Geometrically, this amounts to the complexity of intersections of surfaces as compared to curves.

\subsection{Computation of some nimbers} \label{subsec:nimbring}

Considered separately the game outcome of a Noetherian commutative ring $R$ might be regarded as of minor importance. The proofs of these statements in the previous subsections are more interesting, since they indicate how $R$ is built up from smaller quotients. The real point of interest is the nimber of $R$, because it is a much finer ordinal invariant and it gives a complete description of the game. Also, it is necessary to know the nimber of a game, not just its outcome, when it is part of a sum of games. According to the general definition of the nimber of a combinatorial game (Remark \ref{SG}), the nimber $\alpha(R)$ of a Noetherian commutative ring $R$ is recursively defined by
\[\alpha(R) = \mex \{\alpha(R/\langle a \rangle) : 0 \neq a \in R \}.\]
Unfortunately, the computation of nimbers is much more complicated as we have already seen for abelian groups in Section \ref{subsec:nimb}, because it requires the computation of \emph{all} options and their nimbers. In contrast, in order to show that some Noetherian commutative ring is $\N$ we just have to find \emph{one} option which is $\P$. This explains why the results in this subsection are restricted to rather elementary examples. We do not know the nimber of $K[X,Y]$, but we conjecture that it is quite large. We also conjecture that every ordinal number arises as the nimber of a Noetherian commutative ring.
 
\begin{ex}
We have $\alpha(0)=0$. If $R$ is a field, then $\alpha(R)=1$. Since the trivial ring $0$ is the only normal $\P$-position, we have $\alpha(R)>0$ for all $R \neq 0$.
\end{ex}

\begin{rem}
One can show by induction that $R$ is a mis\`{e}re $\P$-position if and only if $\alpha(R)=1$. This is a general feature of games for which every non-terminal position has a move to a terminal position.
\end{rem}
 
\begin{ex}
If $R$ is a principal ideal domain and $0 \neq r \in R$, then we have $\alpha(R/\langle r \rangle) = \sum_{p \mid r} v_p(r) \eqqcolon \Omega(r)$, where $p$ runs through all prime elements of $R$ up to units and $v_p(r)$ denotes the multiplicity of $p$ in $r$. The proof is analogous to Lemma \ref{Zn}. In particular, $\alpha(R/p^n)=n$ holds for all $n \geq 0$.
\end{ex}

\begin{ex}
From the previous example one may deduce that $\alpha(R)=\omega$ where $R$ is a principal ideal domain which is not a field. For example, we have $\alpha(\Z)=\omega$.
\end{ex}

\begin{ex}
The game of a product of commutative rings $R_1 \times \cdots \times R_n$ is the selective compound of the games of the commutative rings $R_1,\dotsc,R_n$ (Example \ref{ringprod}). This makes it rather easy  in specific examples to determine the nimber of a product. For example, if $R,S$ are two principal ideal domains and $0 \neq r \in R$, $0 \neq s \in S$, then an induction shows that $\alpha(R/\langle r \rangle \times S/\langle s\rangle) = \Omega(r) + \Omega(s)$. From this and another induction we obtain $\alpha(R \times S/\langle s \rangle) = \omega + \Omega(s)$ at least if $R$ is not a field. If also $S$ is not a field, we may further deduce $\alpha(R \times S) = \omega + \omega$. For example, we have $\alpha(\Z \times \Z)=\omega+\omega$.
 
However, it is not always true that $\alpha(R \times S) = \alpha(R)+\alpha(S)$. In fact, in general there is no formula which computes $\alpha(R \times S)$ from $\alpha(R)$ and $\alpha(S)$. Consider the following example: Let $K$ be a field and $R \coloneqq K[X,Y]/\langle X^2,XY,Y^2\rangle$. Then we have $\alpha(R)=\alpha(K)=1$ (Corollary \ref{qu}), but one can verify $\alpha(K \times K)=2$ and $\alpha(R \times K)=4$. Actually this example coincides with the one in Remark \ref{noform}.
\end{ex}
  


\begin{thebibliography}{ANW07}
\bibitem[ANW07]{ANW} M.\ Albert, R.\ Nowakowski and D.\ Wolfe, \emph{Lessons in Play:\ An Introduction to Combinatorial Game Theory}, A K Peters Ltd, USA, 2007
\bibitem[AH87]{AH} M.\ Anderson, F.\ Harary, Achievement and avoidance games for generating abelian groups, \emph{International Journal of Game Theory}, Vol.\ 16, No.\ 4, 321--325, 1987
\bibitem[AM69]{AM69} M.\ Atiyah, I.\ G.\ Macdonald, \emph{Introduction to Commutative Algebra.} Add\-ison-Wesley Series in Mathematics, Vol.\ 361, Addison-Wesley, Boston, 1969
\bibitem[BES16]{BES} B.\ J.\ Benesh, D.\ C.\ Ernst and N.\ Sieben, Impartial avoidance games for generating finite groups, \emph{North-Western European Journal of Mathematics}, Vol.\ 2, 83--101, 2016
\bibitem[BCG01]{BCG} E.\ R.\ Berlekamp, J.\ H.\ Conway and R.\ K.\ Guy, \emph{Winning Ways for Your Mathematical Plays.\ Volume 1}, Second Edition, A K Peters Ltd, USA, 2001
\bibitem[B70]{B70} G.\ J.\ Binder, The two-element bases of the symmetric group, \emph{Izvestiya Vysshikh Uchebnykh Zavedenii. Matematika}, No.\ 1, 9--11, 1970
\bibitem[B91]{B} A.\ J.\ Berrick, Torsion generators for all abelian groups, \emph{Journal of Algebra}, Vol.\ 139, No.\ 1, 190--194, 1991
\bibitem[BS81]{BS} S.\ N.\ Burris, H.\ P.\ Sankappanavar, \textit{A Course in Universal Algebra}, Graduate Texts in Mathematics, Vol.\ 78, First Edition, Springer, New York, 1981
\bibitem[C00]{C} J.\ H.\ Conway, \emph{On Numbers And Games}, AK Peters, CRC Press, USA, 2000
\bibitem[F93]{F93} B.\ Fine, Classification of finite rings of order $p^2$, \emph{Mathematics Magazine}, Vol.\ 66, No.\ 4, 248--252, 1993
\bibitem[FT91]{FT} A.\ Fröhlich, M.\ J.\ Taylor, \emph{Algebraic Number Theory}, Cambridge Studies in Advanced Mathematics, Vol.\ 27, Cambridge University Press, Cambridge, 1991
\bibitem[GW10]{GW} U.\ Görtz, T.\ Wedhorn, \emph{Algebraic geometry.\ Part I:\ Schemes.\ With Examples and Exercises}, Vieweg\,+\,Teubner, 2010
\bibitem[H06]{H} G.\ Hessenberg, Grundbegriffe der Mengenlehre, In:\ \emph{Abhandlungen der Fries'schen Schule}, Neue Folge, Bd.\ 1, 1906
\bibitem[HBE02]{HBE} H.\ U.\ Besche, B.\ Eick, E.\ A.\ O'Brien, A millennium project:\ constructing small groups, \emph{International Journal of Algebra and Computation}, Vol.\ 12, 623--644, 2002
\bibitem[IZ95]{IZ95} I.\ M.\ Isaacs, T.\ Zieschang, Generating Symmetric Groups, \emph{The American Mathematical Monthly}, Vol.\ 102, No.\ 8, 734--739, 1995
\bibitem[K92]{K92} A.\ W.\ Knapp, \emph{Elliptic curves}, Mathematical Notes, Vol.\ 40, Princeton University Press, Princteon, 1992
\bibitem[L02]{L} S.\ Lang, \emph{Algebra}, Graduate Texts in Mathematics, Vol.\ 221, Revised Third Edition, Springer, New York, 2002
\bibitem[S13]{S13} A.\ N.\ Siegel, \emph{Combinatorial Game Theory}, Graduate Studies in Mathematics, Vol.\ 146, Americal Mathematical Society, 2013
\bibitem[S66]{S} C.\ A.\ B.\ Smith, Graphs and composite games, \emph{Journal of Combinatorial Theory}, Vol.\ 1, 51--81, 1966
\bibitem[T87]{T87} R.\ Telg\'{a}rsky, Topological games:\ On the 50th anniversary of the Banach-Mazur game, \emph{Rocky Mountain Journal of Mathematics}, Vol.\ 17, 227--276, 1987
\bibitem[W05]{W} M.\ Wild, Groups of order sixteen made easy, \emph{The American Mathematical Monthly}, Vol.\ 112, No.\ 1, 20--31, 2005
\end{thebibliography}
\end{document}